\documentclass[11pt,reqno]{amsart}
\usepackage{color,enumerate,multirow,epsfig,amsmath,amssymb,amsthm}
\usepackage[longnamesfirst]{natbib}
\usepackage[isolatin]{inputenc} 
\usepackage{a4wide}

\newcommand{\eqdef}{\ensuremath{\stackrel{\mathrm{def}}{=}}}

\def\Cset{{\mathbb{C}}} 
\def\Nset{{\mathbb{N}}} 
\def\Rset{{\mathbb{R}}} 
\def\Zset{{\mathbb{Z}}} 

\def\boldchi{{\boldsymbol\chi}}
  
 \def\mcf{\mathcal{F}} 
\def\mcg{\mathcal{G}}

 \def\rmi{{\mathrm i}} \def\expe{{\mathrm e}}
\def\Floc{\mcf_{\mathrm{local}}}
\def\Fglo{\mcf_{\mathrm{global}}}
\newcommand{\ei}[1]{\expe^{\rmi #1}} \newcommand{\emi}[1]{\expe^{-\rmi #1}}
 \newcommand{\var}{\mathrm{var}}
\newcommand{\cov}{\mathrm{cov}} \newcommand{\eq}[1]{(\ref{eq:#1})}
\newcommand{\prob}{\mathbb{P}} 
\def\esp{{\mathbb{E}}}
\def\ind{{\bf 1}_} 
  \theoremstyle{plain}
\newtheorem{theo}{Theorem}
\newtheorem{coro}[theo]{Corollary} \newtheorem{lem}[theo]{Lemma}
\newtheorem{prop}[theo]{Proposition}

\theoremstyle{remark} \newtheorem*{rem}{Remark}

\newcommand{\ba}{\begin{array}[1]} \newcommand{\ea}{\end{array}}

\newcommand{\limn}{\lim_{n \rightarrow \infty}}

\def\bnu{{\boldsymbol\nu}}  
 \def\bC{{\mathbf C}} \def\bF{{\mathbf F}} \def\bI{{\mathbf
I}}  \def\bS{{\mathbf S}} 
\def\bU{{\mathbf U}} \def\bV{{\mathbf V}} \def\bW{{\mathbf W}}  \def\bY{{\mathbf Y}}    \def\bk{{\mathbf k}}
\def\bt{{\mathbf t}}   \def\bx{{\mathbf
x}}  \def\bz{{\mathbf z}}

\def\rmi{{\mathrm i}}  \def\d{\,{\mathrm d}}
\numberwithin{equation}{section}

\newcounter{hyp}
\newenvironment{hyp}[1]{\refstepcounter{hyp}\begin{itemize}\item[{\bf
(A\arabic{hyp})}] \label{hyp:#1}}{\end{itemize}}
\newcommand{\refhyp}[1]{{\bf (A\ref{hyp:#1})}}
\newcounter{fms}
\newenvironment{fms}[1]{\refstepcounter{fms}\begin{itemize}\item[{\bf
(B\arabic{fms})}] \label{fms:#1}}{\end{itemize}}
\newcommand{\reffms}[1]{{\bf (B\ref{fms:#1})}}

\begin{document}

\title{Moment bounds for non-linear functionals of the periodogram}
\date{\today} 
\author{Gilles Fa\"y}
\address{Laboratoire Paul-Painlevé, Université Lille-1, 59655 Villeneuve-d'Ascq
  Cedex, France.}
\curraddr{APC, Université Paris-7, 10, rue Alice Domon et
  Léonie Duquet, 75205 Paris Cedex 13, France}
\email{gilles.fay@univ-lille1.fr}
\thanks{I am very grateful and indebted to Eric Moulines and Philippe Soulier for their help
  and the many
  fruitful discussions we had on this subject some years ago.}

\begin{abstract} In this paper, we prove the validity of the Edgeworth
expansion of the Discrete Fourier transforms of some linear time series. This
result is applied to approach moments of non linear functionals of the
periodogram.  As an illustration, we give an expression of the mean square
error of the Geweke and Porter-Hudak estimator of the long memory parameter. We prove that this
estimator is rate optimal, extending the result of
\cite{giraitis:robinson:samarov:1997} from Gaussian to linear processes.
\end{abstract}

\maketitle


\section{Introduction} Many estimators in time series analysis involve
non-linear functionals of the periodogram.  Examples include the estimation of
the innovation variance \citep{chen:hannan:1980, lee:cho:kim:park:1995,
  deo:chen:2000, ginovian:2003}, log-periodogram regression
\citep{taniguchi:1979,taniguchi:1991,shimotsu:phillips:2002}, robust
non-parametric estimation of the spectral density
\citep{vonsachs:1994,janas:vonsachs:1995}. Non-linear functionals of the
periodogram also play a predominant role in the analysis of long-memory
time-series: one of the much widely used estimator of the memory parameter is
based on the regression of the log-periodogram ordinates on the log-frequency
\citep*[see also \citealt*{robinson:1995a,moulines:soulier:1999}]%
{geweke:porter-hudak:1983}.

The statistical analysis of such functionals has proved to be a very
challenging problem, due to the intricate dependence structure of periodogram
ordinates. The first attempts to study these statistics were made under the
additional assumption that the underlying process is Gaussian.  Because the
Fourier transform coefficients are in this case also Gaussian, one may then
apply results on non-linear transforms of Gaussian random variables; see for
example \cite{taqqu:1977}, \cite{taniguchi:1980} and \cite{arcones:1994}.

These techniques do not extent to non-Gaussian processes.  A first step to
weaken this assumption was taken by \cite{chen:hannan:1980} who proved the
consistency of an additive functional of the log-periodogram of a linear
stationary process, with an application to the estimation of the innovation
variance.  These techniques were based on the so-called \cite{bartlett:1955}
expansion; this technique was later improved by
\citet*{fay:moulines:soulier:2002} who proved a central limit theorems for
these functionals. It used by\cite{velasco:2000} to establish the weak
consistency of the log-periodogram regression estimate of the long memory
parameter for long range dependent linear time series. Edgeworth expansions are
used to estimates moments of the functional of the unobservable periodogram of
the innovation sequence. Remainder terms can be bounded in probability. The
Bartlett expansion is indeed useful to establish limit theorems but does not in
general allow to determine the moments of these functionals.

An alternative approach has been considered by
\cite{vonsachs:1994,janas:vonsachs:1995}. These authors prove the mean-square
consistency of general additive functional of non-linear transforms of the
(tapered) periodogram, using Edgeworth expansions of the discrete Fourier
transform of the observed time series itself. \cite{janas:vonsachs:1995} apply
these results to prove the mean-square consistency of an Huberized (peak
insensitive) non-parametric spectral estimator. 
These results rely on the Edgeworth expansion of a triangular array of strongly
mixing process with geometrically mixing coefficient established by
\cite{gotze:hipp:1983}. The mixing conditions herein are rather stringent, and
thus the conclusions reached by \cite{janas:vonsachs:1995} are proved under a
set of restrictive assumptions, precluding for instance their use in a
long-memory context.

The main objective of this paper is to develop a method allowing to compute the
moments of functionals of non-linear transforms of the (possibly tapered)
periodogram of a \emph{linear} process. These results are based on Edgeworth
expansion of a (possibly infinite) triangular array of i.i.d. random variables
obtained earlier in \cite{fay:moulines:soulier:2004} and recalled in
Appendix~\ref{sec:edegw-expans-triang}. The linearity of the process is then
crucial. Our results cover both short-memory and long-memory processes.

The remaining of the paper is organized as follows.  In
Section~\ref{sec:notat-assumpt} we give the assumptions on the linear structure
of the time series and define the cumbersome notations related to Edgeworth
expansions.
In Section~\ref{sec:momentboundsshort}, we formulate the validity of Edgeworth
expansions and moment bounds under short memory set of hypotheses. As an
application, we derive the mean-square consistency of additive functionals of
non-linear transform of the periodogram for a short-memory linear time-series.
In Section~\ref{sec:momentboundslong}, we follow the same lines but in a
long-range dependence framework, and apply the moment bounds we obtained to
control the mean-square error of the \cite{geweke:porter-hudak:1983} estimator
of the fractional difference parameter for a non-Gaussian linear long-memory
process. This extends the rate optimality property of the Geweke and
Porter-Hudak (hereafter, GPH) estimator obtained earlier by
\citet*{giraitis:robinson:samarov:1997} for Gaussian processes. A small
Monte-Carlo experiment is run to confirm our results for finite-sample
observations.
Proofs are postponed to the appendices.

\section{Notations and assumptions}
\label{sec:notat-assumpt} Assume that $X= (X_t)_{t \in \Zset}$ is a covariance
stationary process that have a spectral density $f$.  For any integer $r \geq
0$, we define the tapered discrete Fourier transform (DFT) and periodogram of
order $r$ as
\begin{equation}
\label{eq:hurvichtaper} d_{r,n}(\lambda) \eqdef (2\pi n a_r)^{-1/2}
\sum_{t=1}^n h_{t,n}^r X_t \ei{t\lambda}, \ \ I_{r,n}(\lambda) \eqdef |
d_{r,n}(\lambda)|^2
\end{equation} where $h_{t,n} \eqdef 1 - \expe^{2 \rmi \pi t / n}$ is the data
taper introduced in \cite{hurvich:chen:2000} and $a_r\eqdef n^{-1} \sum_{t=1}^n
| h_{t,n}|^{2r} = \binom{2r}{r}$ is a normalization factor. Denote $ d_{r,n,k}
= d_{r,n}(\lambda_k)$ and $I_{r,n,k} = I_{r,n}(\lambda_k)$ the tapered DFT and
tapered periodogram evaluated at the Fourier frequencies $\lambda_k \eqdef
\frac{2\pi k}{n}$ , $k=1,\dots,[(n-1)/2]$.  Define for $r \in \Nset$,
$D_{r,n}(\lambda)$ the normalized kernel function
\begin{equation}
\label{eq:form:kernel} D_{r,n}(\lambda) \eqdef (n a_r)^{-1/2}
\sum_{t=1}^n h_{t,n}^r \exp(i t \lambda) = ( n a_r)^{-1/2} \sum_{k=0}^{r}
\binom{r}{k} (-1)^k D_{n}(\lambda+\lambda_k)
\end{equation} where $D_{n}(\lambda) \eqdef \sum_{t=1}^n \emi{\lambda t}$
denotes the non-symmetric Dirichlet kernel. The latter relation implies that
$D_{r,n}(\lambda_k)= 0$ for $k \in \{ 1, \cdots, \tilde n \}$, with $\tilde n
\eqdef \lfloor(n-2r-1)/2\rfloor$, so that the tapered Fourier transform is
invariant to shift in the mean.  As shown in \cite{hurvich:chen:2000}, the
decay rate of the kernel in the frequency domain increases with the kernel
order, namely
\begin{equation}
\label{eq:decaytaper}
   \forall \lambda \in [-3\pi/2,3\pi/2]\ , \
|D_{r,n}(\lambda)| \leq \frac{C n^{1/2}}{(1+n|
\lambda|)^{r+1}}
\end{equation}
This property means that higher order kernels are more effective to
control frequency leakage.  If $X$ is a white noise and $r=0$, the DFT
ordinates at different Fourier frequencies are uncorrelated.  This property is
 lost by tapering.  More precisely, for $1 \leq k \ne j \leq
\tilde{n}$, $ \esp[ d_{r,n,k} d_{r,n,j}] = 0$, and $\esp[d_{r,n,k}
\overline{d_{r,n,j}} ] \eqdef (2\pi)^{-1}\varsigma_r(k-j)$ where $\bar z$
denotes the complex conjugate of $z$ and $\varsigma_r$ defined
in~(\ref{eq:defvarsigma}).

Many statistical applications (see the references given in the Introduction)
require to study weighted sums of non-linear functionals of the periodogram
ordinates
\begin{align}
\label{eq:defTnrenorm} T_n(X,\phi) = \sum_{k=1}^{K} \beta_{n,k} \phi
\left(\frac{I_{r,n,k}}{f(\lambda_k)}\right) , 
\end{align} where $(\beta_{n,k})_{k \in \{1, \dotsc, K \}}$ is a triangular
array of real numbers.  If $X$ is a Gaussian white noise, then $(I_{r,n,k})$
are i.i.d and the moments of the sum $T_n(X,\phi)$ can be calculated
explicitly.  In any other case, the random variables $( I_{r,n,k})_{k \in \{1,
  \dotsc, K\}}$ are not independent, and the calculation of the moments of
$T_n(X,\phi)$ is a difficult problem.  The only attempt to solve it has been
made by \cite{janas:vonsachs:1995}, who proposed a technique to compute moment
of order 1 and 2. As already outlined, their results are based on mixing
conditions, precluding their use for long-memory processes.
\begin{rem}
  Sometimes the periodogram ordinates are averaged along blocks of adjacent
  frequencies. This technique is known as \emph{pooling} and is appropriate to
  reduce asymptotic variance of the estimators of non linear functionals of the
  periodogram (see \citealp*{robinson:1995a,robinson95b}). For simplicity, we
  will not present any explicit result or application with the pooled
  periodogram, but the Edgeworth expansion results that follow allow to derive
  moment bounds on functionals of tapered \emph{and} pooled periodogram as
  well.
\end{rem}

In this contribution, we focus on non-Gaussian strict sense linear processes,
\textit{i.e.}  it is assumed that
\begin{align}
\label{eq:Xlinear} X_t = \sum_{j\in\Zset} \psi_j Z_{t-j}, \quad \sum_{j \in
\Zset} \psi_j^2 < \infty \; ,
\end{align} where $(Z_j)_{j \in \Zset}$ is a sequence of i.i.d random variables
such that $\esp[Z_1]=0$, $\esp[Z_1^2] = 1$.  In addition, for some $s \geq 3$,
$p\geq1$ and $p'\geq 0$,
\begin{hyp}{Lp} $\esp[|Z_1|^{s}]<\infty$ and $\int_\Rset|t|^{p'}\ |\esp[\ei{t
Z_1}]|^{p}\d t < \infty$ .
\end{hyp}
\begin{rem}
  Apart from a classical moment condition, \refhyp{Lp} suppose that the
  distribution of the i.i.d. noise is smooth; for example, lattice
  distributions are forbidden. This condition is stronger than the usual
  Cram\'er condition. It ensures that the distributions of the Fourier coefficients
  of $Z$ are eventually continuous. We need this continuity to bound moments of
  singular functionals of the periodogram. Note that this condition
  could be dispensed with, were we concerned with smooth functionals.
\end{rem}

Define $\psi(\lambda) = \sum_{j \in \Zset} \psi_j \ei{j \lambda}$ (the
convergence holds in $\mathbb{L}^2([-\pi,\pi],\d x)$) the transfer function of the
linear filter $(\psi_j)_{j \in \Zset}$ and $f(\lambda) = (2 \pi)^{-1}
|\psi(\lambda)|^2$ the spectral density of the process $X$.  For an integer $k \in
\{1,\cdots,\tilde n\}$ such that $f(\lambda_k) \neq 0$, define the normalized
DFT $ \omega_{r,n,k} \eqdef \sqrt{2 \pi} \; d_{r,n,k}/ |\psi(\lambda_k)|$.  Let
$k_1 < k_2 < \ldots < k_u $ be an ordered $u$-tuple of such integers in the range
$1,\ldots,\tilde n$ and write $\bk=(k_1,\ldots,k_u)$.  Define (the
reference to $r$ is suppressed in the notation)
\begin{equation}\label{eq:defsnk} \bS_{n}(\bk) \eqdef \left[ \mathrm{Re}
(\omega_{r,n,k_1}), \mathrm{Im} (\omega_{r,n,k_1}), \cdots, \mathrm{Re}
(\omega_{r,n,k_u}), \mathrm{Im} (\omega_{r,n,k_u}) \right] .
\end{equation} With those definitions,
\begin{equation}
  \label{eq:relationSnIn}
  I_{n,k,r}= f(\lambda_k) |\omega_{r,n,k}|^2 = f(\lambda_k) \|\bS_n(k)\|^2 \; .
\end{equation}
Since $X$ admits the linear representation \eq{Xlinear},
$\bS_n(\bk)$ can be further expressed as a $2u$-dimensional infinite triangular array
in the variables $(Z_t)_{t \in \Zset}$. Precisely
\begin{equation}
\label{eq:bSn} \bS_{n}(\bk) = \sum_{j \in \Zset} \bU_{n,j}(\bk) Z_j,
\end{equation} with
\begin{align} \bU_{n,j}(\bk) & \eqdef (n a_r)^{-1/2}
\bF_n^{-1}(\bk)
\sum_{t=1}^{n} \psi_{t-j} \bC_{n,t}(\bk), \label{eq:defUnj} 
\\ \bC_{n,t}(\bk)
& \eqdef \sum_{p=0}^r (-1)^p \binom{r}{p} \Bigl(
\cos(t\lambda_{k_1+p}),\sin(t\lambda_{k_1+p}), \dots, \cos(t\lambda_{k_u+p}),
\sin(t\lambda_{k_u+p})\Bigr)' \nonumber \\ \text{and} \quad 
\bF_n(\bk) & \eqdef \mathrm{diag}\bigl(|\psi(\lambda_{k_1})|, |\psi(\lambda_{k_1})| , \dots,
|\psi(\lambda_{k_u})|, |\psi(\lambda_{k_u})|\bigr) \; . \nonumber 
\end{align} To formulate our results, some notations related to Edgeworth
expansions are required, which we take from the monograph of
\citet*{bhattacharya:rao:1976}.  For $u$ a positive integer,
$\bnu=(\nu_1,\ldots,\nu_u)\in\Nset^u$ and $\bz=(z_1,\ldots,z_u)\in\Cset^u$,
denote $|\bnu| = \sum_{i=1}^u \nu_i$, $\bnu!  = \nu_1! \nu_2! \cdots \nu_u!$
and $\bz^\bnu=z_1^{\nu_1}z_2^{\nu_2}\cdots z_u^{\nu_u}$.  If $1 \leq |\bnu|
\leq s$, denote $\chi_{n,\bnu}(\bk)$ the cumulants of $\bS_n(\bk)$. Then $
\chi_{n,\bnu}(\bk) = \kappa_{|\bnu|} \sum_{j \in \Zset} \bU_{n,j}^\bnu(\bk) $
where $\kappa_r$ denotes the $r$-th cumulant of $Z_1$, $r \leq s$.  Let
$\bV_n(\bk) \eqdef \cov[\bS_n(\bk)] = \sum_{j\in\Zset}
\bU_{n,j}(\bk)\bU'_{n,j}(\bk).$ Let $\boldchi = \{\chi_{\bnu}; \ \bnu\in\Nset^u
\}$ be a set of real numbers. For any integer $r\geq2$ and $\bz\in\Cset^u$,
define $\chi_{r}(\bz) \eqdef r!\sum_{|\bnu|=r} \frac{\chi_{\bnu} \;
\bz^\bnu}{\bnu!}$.  The polynomials $\tilde{P}_r(\bz, \boldchi )$ are formally
defined for $r\geq1$ by the identities
\begin{equation*} 1 + \sum_{r=1}^{\infty} \tilde{P}_r(\bz,\boldchi) t^r = \exp
\Bigl \{\sum_{r=3}^{\infty} \frac{\chi_{r}(\bz)}{r!} t^{r-2} \Bigr \} = 1 +
\sum_{m=1}^\infty \frac{1}{m!} \Bigl ( \sum_{r=3}^{\infty}
\frac{\chi_{r}(\bz)}{r!} t^{r-2} \Bigr)^m,
\end{equation*} and we set $\tilde{P}_0 \equiv 0$.  Denote $\varphi_\bV$ the
density of a Gaussian r.v in $\Rset^u$ with zero mean  and non-singular
covariance matrix $\bV$. 
Define $P_r: \Rset^u \mapsto \Rset$ by $ P_r(\bx,\bV,\boldchi) =
\left[\tilde{P}_r(-D,\boldchi)\right] \varphi_\bV(\bx) $ where, for any
polynomial $P(\bz) = \sum_{\bnu} a_{\bnu} \bz^{\bnu}$, $P(-D)$ is interpreted
as a polynomial in the differentiation operator $D$, $P(-D) = \sum_{\bnu}
a_\bnu (-1)^{|\bnu|}D^{\bnu}$, with $ D^\bnu =
\frac{\partial^{|\bnu|}}{\partial x_1^{\nu_1}\ldots \partial x_u^{\nu_u}}$,
$\bnu = (\nu_1, \dots, \nu_u) \in \Nset^u$.  By construction $P_r$ and $\tilde
P_r$ do not depend on the coefficient $\chi_\bnu$ if $|\bnu| > r+2$, and
$\tilde{P}_r(\rmi \bt,\boldchi)\expe^{-\bt'\bV\bt/2}$ is the Fourier transform
of $P_r(\bx,\bV,\boldchi)$.
Let $\xi_\Gamma$ be a centered $a$-dimensional Gaussian vector with covariance
matrix $\Gamma$ and $g : \Rset^a \to \Rset$ a measurable mapping.  Define
$N_{s}(g) = \int_{\Rset^{a}} (1 + \|\bx\|^s)^{-1} |g(\bx)| \d\bx$ and
$\|g\|^2_\Gamma = \esp[g^2(\xi_\Gamma)]$.  The Hermite rank of
$g$, $\|g\|^2_\Gamma < \infty$, with respect
to $\Gamma$ is defined as the smallest integer $\tau$ such that there exists a
polynomial $P$ of degree $\tau$ with $\esp[g(\xi_\Gamma)P(\xi_\Gamma)] \neq 0$.
We denote $\tau(g,\Gamma)$ the (positive) Hermite rank of $g -
\esp[g(\xi_\Gamma)]$ with respect to $\Gamma$. 

\section{Moment bounds: short memory case}
\label{sec:momentboundsshort}
In this section we consider  short-range dependent
processes.  For any reals $\alpha,\delta > 0$ and $\beta < \infty$, denote by $\mathcal
G(\alpha,\beta,\delta)$ the set of real sequences $(\psi_j)_{j \in \Zset}$ such that
\begin{align}
  \label{eq:weakdependence}
  &| \psi_0 | + \sum_{j\in\Zset} |j|^{1/2+\delta}|\psi_j| \leq \beta \; ,\\
  \label{eq:spectraldensity}
  & \alpha \leq \inf_{\lambda \in [-\pi,\pi]} \bigl| \psi(\lambda) \bigr| \; .
\end{align}
\begin{theo}  \label{theo:edgdftshort}
Assume  \refhyp{Lp} with some integer $s \geq 3$, $p\geq 1$ and $p'=0$ and
assume that  $(\psi_j)_{j \in \Zset} \in \mathcal G(\alpha,\beta,\delta)$ for some $\alpha,\delta > 0$
  and $\beta < \infty$.  Then, there exists constants $C$ and $N$ (depending only on $s,p,\alpha,
  \beta,\delta, u$ and the distribution of $Z_0$) such that, for all $n \geq N$, and
  all $u$-tuple $\bk$ of distinct integers, the distribution of
  $\bS_{n}(\bk)$ has a density $q_{n,\bk}$ with respect to Lebesgue's measure
  on $\Rset^{2u}$ and
\begin{align}
  \sup_{\bx \in \Rset^{2u}} (1 + \|\bx\|^s)\bigl|q_{n,\bk}(\bx) -
  \sum_{r=0}^{s-3} P_r(\bx,\bV_n(\bk),\{\chi_{n,\bnu}(\bk)\}) \bigr| \leq
  Cn^{-(s-2)/2} \; .
  \label{eq:devedgshort}
\end{align}
\end{theo}
Several interesting consequences can be derived from this result. A
straightforward integration of the expansion~(\ref{eq:devedgshort}) yields the
following corollary which gives an Edgeworth expansion of some moment $
\esp[g(\bS_n(\bk))]$ around the centered Gaussian distribution with covariance
matrix $\bV_n(\bk)$.
\begin{coro}
\label{coro:boundmom1}
There exists a constant $C$ and an integer $N$ (depending only on
$s,p,\alpha,\beta,\delta, u$ and the distribution of $Z_0$) such that, for any
$u$-tuple of distinct integers $\bk$, $n \geq N$ and measurable function $g$
satisfying $N_{s}(g) < \infty$,
  \begin{align}
    \left| \esp[g(\bS_n(\bk))] - \sum_{r=0}^{s-3} \int_{\Rset^{2u}} g(\bx)
      P_r(\bx,\bV_n(\bk),\{\chi_{n,\bnu}(\bk)\}) \d\bx \right| \leq C \; N_s(g)
    \; n^{-(s-2)/2}.
  \end{align}
\end{coro}
One can also use Theorem~\ref{theo:edgdftshort} to develop the same moment
around the \emph{limiting} Gaussian distribution of $\bS_n$.  Recalling that
$\omega_{r,n,k} = a_r^{-1/2}\sum_{s=0}^r \binom{r}{s} (-1)^s \omega_{0,n,k+s}$,
we have \[\lim_{n \to \infty} \bV_n(\bk) = \bV(\bk)\] under short memory
conditions, where $\bV(\bk)$ is the $2u
\times 2u$ matrix defined component-wise by
\begin{gather}
\label{eq:defbvbk}  [\bV(\bk)]_{2i-1,2j-1} = [\bV(\bk)]_{2i,2j} = \frac 1 2 \varsigma_r(k_i - k_j)
  \; \; , \; \; [\bV(\bk)]_{2i-1,2j} = [\bV(\bk)]_{2i,2j-1} = 0 \; , 
\end{gather}
for $i,j=1,\cdots,u$, with
\begin{align}
\label{eq:defvarsigma}
\varsigma_r(l) \eqdef
\begin{cases}
  0 & \text{if} \ |l| > r \;, \\ a_r^{-1} (-1)^l \binom{2r}{r+l} & \text{if} \ 
  |l| \leq r \;.
\end{cases}
\end{align}
 Note that $\bV(\bk) = \tfrac 12 \bI_{2u}$ if $r=0$.
\begin{coro}
\label{coro:boundmomshort1}
There exists a constant $C$ and $N$ (depending only on $s,p,\alpha$,$\beta$,$\delta$,$u$
and the distribution of $Z_0$), such that for all measurable function $g$ on
$\Rset^{2u}$ such that $N_3(g) < \infty$, all $u$-tuple of distinct integers
$\bk$, and any $n \geq N$,
  \begin{equation}\label{eq:boundmomshort1}
    \left | \esp \left[ g(\bS_n(\bk)) \right] - \int_{\Rset^{2u}} g(\bx)
      \varphi_{\bV(\bk)}(\bx) \d \bx \right| \leq C \left\{ n^{-1/2} N_3(g) +
      n^{-\tau(g,\bV(\bk))/2} \; \| g \|_{\bV(\bk)} \right\}.
  \end{equation}
\end{coro}
For some functions $g$, it is possible to sharpen this result by considering
higher-order ($s > 3$) expansions  and approximating the terms appearing in
these expansions.  We shall consider mappings $g:\Rset^{2u}\to\Rset$ such that  
\begin{gather}
\label{eq:propG} 
  \begin{array}{rl}
&g(x_1,\dots,x_{2u}) = \prod_{j=1}^u  g_j(x_{2j-1},x_{2j})\\  
\text{with} & g_j(x,y) = g_j(y,x) = g_j(-x,y) , \; \; j=1,\dots,u.  
  \end{array}\end{gather} Recalling~(\ref{eq:relationSnIn}), products of
functionals of the periodogram are included in this particular case. Better
bounds are obtained by considering frequencies $k_1,\dots,k_u$ separated by
$r$, so that the asymptotic decorrelation is achieved, $\bV(\bk) = \tfrac 12
\bI_{2u}$ as in the $r=0$ case. Under those conditions, the $O(n^{-1/2})$ of
Corollary~\ref{coro:boundmomshort1} can ben improved to $O(n^{-1})$.
\begin{coro}
\label{coro:boundmomshort2}
Under the hypothesis that $s \geq 4$, there exists a constant $C$ and $N$
(depending only on $s,p,\alpha, \beta, \delta, u$ and the distribution of
$Z_0$), such that for all measurable function $g$ satisfying~(\ref{eq:propG})
and such that $N_q(g) < \infty$, all $u$-tuple of ordered integers $\bk$ such that $k_i
< k_{i+1} - r$, and any $n \geq N$,
  \begin{equation}
    \label{eq:boundmomshort2}
  \left | \esp \left[ g(\bS_n(\bk) \right] - \int_{\Rset^{2u}} g(\bx)
    \varphi_{\bI_{2u}/2}(\bx) \d \bx \right| \leq C \left\{  n^{-(s-2)/2}N_s(g) +
   n^{-1} \| (1+\|\bx\|^s)g(\bx)\|_{\bI_{2u}}  \right\}.
  \end{equation}
\end{coro}
The proofs of Corollaries~\ref{coro:boundmomshort1}
and~\ref{coro:boundmomshort2} are postponed to the
Appendix~\ref{sec:proofs-coroll}. 
\begin{rem}
Pushing to higher orders $s \geq 4$ in
Corollary~\ref{coro:boundmomshort2} is sometimes necessary to have $N_s(g) < \infty$ (see
the applications below). But it does not improve the $O(n^{-1})$ bound. 
\end{rem}
To illustrate the results above, we compute bounds for the mean-square error of
plug-in estimators of non-linear functionals of the spectral density $
\Lambda(f) = \int_0^{\pi} w( \lambda) G(f(\lambda ))\d \lambda $
 where $w$ is a function of bounded variation and $G$ is a function such
 that there exists a function $H$ satisfying, for any $x > 0$, $ \int_0^\infty
 |H(xv)| \expe^{-v} \d v < \infty$ and $\int_{v>0} H(xv)\expe^{-v}\d v = G(x)
 ,$ \textit{i.e.} $H$ is the inverse Laplace transform of the function $t \mapsto
 G(1/t)/t$.
We consider the following estimator
$$
\hat \Lambda_n = (\pi/\tilde n ) \sum_{k=1}^{\tilde n} w(\lambda_k)
H(I_{n,k})
$$
and put $\Lambda_n = (\pi/\tilde n ) \sum_{k=1}^{\tilde n} w(\lambda_k)
G(f(\lambda_k))$. Here, $r=0$ and $I_{n,k} \eqdef I_{0,n,k}$ is the ordinary
periodogram.  We assume that the approximation error $\Lambda_n - \Lambda$ is
neglectable in comparison with the mean-square error $\esp( \hat \Lambda_n -
\Lambda_n)^2$.  These functionals have been studied in \cite{taniguchi:1980} in
the Gaussian case and \cite{janas:vonsachs:1995} for non-Gaussian linear
process, under rather stringent assumptions \citep[see also][and the references
therein]{deo:chen:2000} .  The moment bounds we have established allow to
extend \citet*{janas:vonsachs:1995}'s result, by relaxing the conditions on the
dependence (from $|\psi_j| < C \rho^{|j|}$ for some $\rho \in (0,1)$ to
$\sum_{j\in\Zset} |j|^{1/2}|\psi_j| < \infty$).
\begin{prop}
  \label{prop:boundNFL}
  Let $(X_t)_{t\in \Zset}$ be sequence satisfying the assumptions of
  Theorem~\ref{theo:edgdftshort} with some $s \geq 4$.  Put $H_1(x_1,x_2) =
  H(x_1^2 + x_2^2)$, $H_2(x_1,x_2,x_3,x_4) = H_1(x_1,x_2)H_1(x_3,x_4)$ and
  assume that $N_3(H_1^2) < \infty$ and $N_5(H_2) < \infty$. Then,   uniformly in $f \in \mcg(\alpha,\beta,\delta)$
\begin{align*}
  \esp[(\hat\Lambda_n - \Lambda_n)^2] \leq Cn^{-1} \ .
\end{align*}
\end{prop}
\begin{proof}[Sketch of the proof]
Applying Corollary~\ref{coro:boundmomshort1} to the function $g_{k,f}(x_1,x_2) =
H[f(\lambda_k)(x_1^2 + x_2^2)])$ and Corollary~\ref{coro:boundmomshort2} to $g_{k,j,f}(x_1,x_2,x_3,x_4) =
H[f(\lambda_k)(x_1^2 + x_2^2)]H[f(\lambda_j)(x_3^2 + x_4^2)]$ yield asymptotic
expansions for the moments  $\esp[H^2(I_{n,k})]$ and
$\esp[H(I_{n,k})H(I_{n,j})]$, which are sufficient to derive the result. The
uniformity of the constant $C$ follows from the existence of bounds on
$N_3(g_{k,f})$  and $N_4(g_{k,j,f})$ which are uniform in $\psi\in\mcg(\alpha,\beta,\delta)$. 
\end{proof}

\section{Moment bounds : Long memory case}
\label{sec:momentboundslong}
\subsection{Assumptions and main results}
We consider two sets of assumptions, depending on available information on the
behavior of the spectral density outside a neighborhood of the zero frequency.
Recall that a real valued
function $\phi$ defined in a neighborhood of zero is regularly varying at zero
with index $\rho\in\Rset$ if, for all $x$ and all $t>0$, $\lim_{x\to0}
\phi(tx)/{\phi(x)} = t^\rho$. If $\rho=0$, the function $\phi$ is said slowly
varying at zero.
Let $\vartheta\in(0,\pi)$, $0 < \delta < 1/2$, $\Delta < \delta$.  We say that
the linear filter $( \psi_j )_{j \in \Zset} $ belongs to the set
$\mcf(\vartheta,\delta,\Delta,\mu)$ if $\sum_{j=-\infty}^{\infty} \psi_j^2 <
\infty$ and if there exists $d \in [\Delta,\delta]$ such that $\psi(\lambda)$
is regularly varying at zero with index $-d$ and that
\begin{gather}
\label{eq:cond1bis}
\frac{\int_0^\pi \lambda^{2d} |\psi(\lambda)|^2 \d \lambda}{\min_{0 \leq |\lambda|
    \leq \vartheta} \lambda^{2d} |\psi(\lambda)|^2} \leq \mu \; , \\
\label{eq:cond2}
\forall j \geq 0, \ \frac{ |\psi_j| + \sum_{|t| \geq j}^\infty |\psi_{t+1} -
  \psi_t |}{\min_{0 \leq \lambda \leq \vartheta} \lambda^d |\psi(\lambda)|}
\leq \mu (1 + j)^{d-1} \; ,
\end{gather}
An example is provided by $\psi(\lambda) \eqdef \left(1 - \ei{\lambda} \right)^{-d}$
the transfer function of the causal fractional integration filter, $\psi_t =
\Gamma(t+d)/(\Gamma(d)\Gamma(t+1)), t \geq 0$.  

\subsubsection*{Local-to-zero assumptions}
We first consider local-to-zero assumptions for which nothing is required
outside a neighborhood of the zero frequency, apart from integrability of the
spectral density (see \cite{robinson:1995a}). For $\beta > 0$, we say that the
sequence $(\psi_j)_{j \in \Zset}$ belongs to the set
$\Floc(\vartheta,\beta,\delta,\Delta,\mu)$ if $( \psi_j )_{j \in
  \Zset} \in \mcf(\vartheta,\delta,\Delta,\mu)$ and
\begin{align}
\label{eq:condloc}
& \forall \lambda \in (0,\vartheta] \ , \ \frac{|\psi^*(\lambda) -
\psi^*(0)|}{\min_{\lambda \in (0,\vartheta]}|\psi^*(\lambda)|} \leq \mu
\lambda^\beta
\end{align}
with $\psi^*(\lambda) = (1- \ei{\lambda})^d \psi(\lambda)$ where $d$ is the
index of regular variation of $\psi$.  This class is quite general and includes
the impulse response of FARIMA filters \citep*[see][and the references
therein]{doukhan:oppenheim:taqqu:2002} but also processes whose spectral
density may exhibit singularity outside the zero frequency, such as the
Gegenbauer's processes. 
As seen below, under local-to-zero assumptions, the validity of the Edgeworth
expansion can only be established for the DFT coefficients in a degenerating
neighborhood of zero frequency. This is enough for, say, semi-parametric
estimation of the long-memory index by the GPH method.
\subsubsection*{Global assumptions}
In some situations, it is possible to formulate regularity assumptions over
full the frequency range $[-\pi,\pi]$ or a subset of it. These assumptions allow to prove the
validity of the Edgeworth expansion for all the frequency ordinates.  We say
that the sequence $( \psi_j ) $ belongs to the set
$\Fglo(\vartheta,\beta,\delta,\Delta,\mu)$ if $(\psi_j) \in
\Floc(\vartheta,\beta,\delta,\Delta,\mu)$ and if in addition, for
all $(\lambda,\lambda') \in (0,\vartheta] \times (0,\vartheta]$,
\begin{align}
\label{eq:condglob}
&\left| \psi^*(\lambda) - \psi^*(\lambda') \right| \leq \mu
\frac{|\psi^*(\lambda)| \vee |\psi^*(\lambda')|}{|\lambda| \wedge |\lambda'|}
|\lambda - \lambda'|
\end{align}

Under those assumptions and as in the short-memory case, we are able to prove
the validity of the Edgeworth expansion for the DFT's
(Theorem~\ref{theo:edgdftlong}) and deduce some moment bounds
(Corollaries~\ref{coro:devedgmomentLM},~\ref{coro:boundmomlong1}
and~\ref{coro:boundmomlong2}). In comparison with short memory results, note
that tapering ($r > 0$) and \refhyp{Lp} with $s \geq p'$ are required.
%
\begin{theo}  \label{theo:edgdftlong}
Assume  \refhyp{Lp} with some integer  $s \geq 3$, $p\geq 1$ and $p' \geq s$. Let $r$
  be a positive integer and $\beta$, $\delta$, $\Delta$, $\mu$, $\vartheta$ be
  constants such that $0 < \delta < 1/2$, $-r+1/2 < \Delta \leq 0$, $\mu > 0$
  and $\vartheta\in(0,\pi]$. Let $(m_n)_{n \geq 0}$ be a non-decreasing
  sequence. Assume either
  \begin{eqnarray}
    \label{cond:local} ( \psi_j )_{j \in \Zset}
    \in \Floc(\vartheta,\beta,\delta,\Delta,\mu)  \qquad \text{and} \qquad
  \lim_{n \to \infty} \Bigl( \frac{1}{m_n} + \frac{m_n}{n} \Bigr) = 0
\end{eqnarray}
or
\begin{equation}
 \label{cond:global} 
 ( \psi_j )_{j \in \Zset}
  \in\Fglo(\vartheta,\beta,\delta,\Delta,\mu)  \qquad \text{and} \qquad m_n \leq \vartheta \tilde{n}.
  \end{equation}
  Then there exist a constant $C$ and positive integers $K_0$, $N_0$ which
  depends only on $\vartheta$, $\beta$, $\delta$,  $\Delta$, $\mu$, the distribution of
  $Z_1$ and the sequence $(m_n)$, such that for any $n\geq N_0$ and
  $\bk=(k_1,\dots,k_u)$ of integers in the range $\{K_0,\dots,m_n\}$, the
  distribution of $\bS_{n}(\bk)$ has a density $q_{n,\bk}$ with respect to
  Lebesgue measure on $\Rset^{2u}$ which satisfies
\begin{align}
  \sup_{\bx \in \Rset^{2u}} (1 + \|\bx\|^s)\bigl|q_{n,\bk}(\bx) -
  \sum_{r=0}^{s-3} P_r(\bx,\bV_{n}(\bk),\{\chi_{n,\bnu}(\bk)\}) \bigr| \leq
  Cn^{-(s-2)/2}. \label{eq:devedglong2}
\end{align}
If $u=1$, one can take  $K_0 = 1$.
\end{theo}
Integrating  some function $g$ against the density $q_{n,\bk}$ and
using~(\ref{eq:devedglong2}) yields the following corollary.
\begin{coro}
\label{coro:devedgmomentLM}
Under the assumptions of Theorem~\ref{theo:edgdftlong}, there exists a constant $C$ and an integer $N$ depending only on $\vartheta$,
$\beta$,$\delta$, $\Delta$, $\mu$, $u$, $r$ and such that, for all $u$-tuple of distinct
integers $\bk$ satisfying $K_0 \leq \min(\bk)$, $\max (\bk) \leq m_n$ and any $n
\geq N$, and all measurable function $g$ such that $N_s(g) < \infty$,
  \begin{align}
    \left| \esp[g(\bS_n(\bk))] - \sum_{r=0}^{s-3} \int_{\Rset^{2u}} g(\bx)
      P_r(\bx,\bV_n(\bk),\{\chi_{n,\bnu}(\bk)\}) \d\bx \right| \leq C \ N_s(g)\
    n^{-(s-2)/2}.
  \end{align}
\end{coro}
Similarly to the short-memory case, one could approximate $\esp[g(\bS_n(\bk))]$
using the limiting distribution of $\bS_n(\bk)$ in place of the Gaussian
approximation as Corollary~\ref{coro:devedgmomentLM}.  Under long-range
dependence and for fixed $\bk$, the limiting covariance matrix of $\bS_n(\bk)$
fully depends on $\bk$ and not only on $(k_2-k_1,\dots,k_u-k_{u-1})$. This
behavior at ``very-low frequencies'' as been studied for instance by
\cite{hurvich93}.  However, one can control the covariance of the
standardized DFT coefficients and then the difference $\bV_n(\bk) -
\bV(\bk)$ thanks to the following lemma.
\begin{lem}\label{lem:covdft}
  For $1\leq k \leq j \leq \vartheta n/\pi - r$ and $r \geq 1$, there exists a
  constant $C$ depending only on $\vartheta,\beta,\delta,\Delta,\mu$ such that
\begin{align}
  \left| \esp ( \omega_{r,n,k}  \omega_{r,n,j} ) \right| +
  \left|\esp(\omega_{r,n,k}\bar \omega_{r,n,j}) - \varsigma_r(k-j) \right|
  \leq C p(k,j,n,\beta)
\label{eq:simple}
\end{align}
with
\begin{align}
  \label{eq:defipkj}
  p(k,j,n,\beta) =
  \begin{cases}
    (jk)^{-1/2} + \left( \frac{j \vee k}{n}
    \right)^\beta  & \text{under} \quad (\ref{cond:local}) \\ (
    jk)^{-1/2}  & \text{under} \quad (\ref{cond:global}).
  \end{cases}
  \end{align}
\end{lem}

Thus, we can develop the moments around the Gaussian distribution with
covariance matrix $\bV(\bk)$ as in the short-memory context. The two following
corollaries prove sufficient for our applications.  The next corollary is
useful for moment bounds on one frequency $k$.
\begin{coro}
  \label{coro:boundmomlong1}
  Under the assumptions of Theorem~\ref{theo:edgdftlong}, there exist a
  constant $C$ and a positive integer $N_0$ which depend only on $\vartheta$,
  $\beta$, $\delta$, $\Delta$, $\mu$, the distribution of $Z_1$
  and the sequence $(m_n)$, such that for any $n\geq N_0$, for any integer $k$
  in the range $\{1,\dots,m_n\}$ and any measurable function $g$ on $\Rset^{4}$
  such that $N_3(g) < \infty$
 \begin{multline*}
 \left | \esp \left[ g(\bS_n(k)) \right] - \int_{\Rset^2} g(\bx) \varphi_{\bI_2/2}(\bx)
   d \bx \right|\\ \leq C \left\{ n^{-1/2} N_3(g) +
   p(k,k,n,\beta)^{-\tau(g,\bV(k))/2} \|g(\bx)\|_{\bI_{2u}} \;
 \right\}.
\end{multline*}
\end{coro}

The next corollary is useful for moment bounds on two frequencies $k<j-r$.
\begin{coro}
  \label{coro:boundmomlong2}
  Under the assumptions of Theorem~\ref{theo:edgdftlong}, there exist a
  constant $C$ and positive integers $K_1 \geq K_0$, $N_0$ which depends only
  on $\vartheta$, $\beta$, $\delta$,  $\Delta$, $\mu$, the
  distribution of $Z_1$ and the sequence $(m_n)$, such that for any $n\geq N_0$
  and for any couple $\bk=(k,j)$ of integers in the range $\{K_0,\dots,m_n\}$
  such that $k < j-r$ and any measurable function $g$ on $\Rset^{4}$
  verifying~(\ref{eq:propG}) and such that $N_4(g) < \infty$
\begin{multline*}
  \left | \esp \left[ g(\bS_n(\bk)) \right] - \int_{\Rset^4} g(\bx)
    \varphi_{\bI_4/2}(\bx) d \bx \right|
 \\ \leq     C    \left\{ n^{-(s-2)/2} N_s(g) +
    n^{-1/2} p^2(k,j,n,\beta)\| (1+\|\bx\|^s)g(\bx) \|_{\bI_4/2} \right\} \ .
 \end{multline*}
\end{coro}
\subsection{GPH estimation of the long memory parameter }
\label{sec:GPH}
\subsubsection{Theoretical results}
\label{sec:theoretical-results}
A very widely used estimator of the memory parameter $d$ was introduced by
\citet*{geweke:porter-hudak:1983}.  It is obtained from the linear regression
of the log-periodogram of the observations using the logarithm of the
frequencies as explanatory variable. In contrast with the Whittle estimator,
the GPH is defined explicitly in terms of the log-periodogram ordinates, see
Eq.~(\ref{eq:GPHexplicit}) below. Many theoretical work has been achieved on
this estimator, in stationary or non-stationary contexts (see
\citealp*{fay:FvW:2007} for a survey of the main results). For instance,
\cite{giraitis:robinson:samarov:1997} proved that the GPH of Gaussian $X$ is
rate optimal for the quadratic risk and over some classes of spectral densities
that is included in our $\Floc$.  To compute the risk of the GPH estimator, one
need to compute or approximate moments of the log-periodogram. The
log-periodogram is a non-smooth function of the Fourier transform of the
observation, which are Gaussian if $X$ is Gaussian. The proof by
\cite{giraitis:robinson:samarov:1997} relies on moment bounds of non-linear
function of Gaussian variables \citep*[see][]{arcones:1994,soulier:2001}; this
technique does not extend naturally to non-Gaussian time series.  Here, we
shall apply the Edgeworth approximations obtained in preceding section to
extend this result to the case of strong sense linear process.

For the sake of simplicity of exposition, we only consider a taper of order $r=1$
and write $I_k = I_{1,n,k}$. The GPH estimator is obtained by an ordinary least
square regression of $\log(I_{k})$ on $\log |2\sin(\lambda_{k}/2)|$
\citep*[see][]{geweke:porter-hudak:1983,robinson:1995a}. With the frequency
spacing and taper order $r$, one regresses on every $r+1$ frequency. For $r=1$ it writes
\[
(\hat d_{m},\hat C) = \arg\min_{d',C'} \sum_{k=1}^m \left\{\log(I_{2k+1}) + 2d'
\log |2\sin(\lambda_{2k+1}/2)| - C'\right\}^2.
\]
where $m=m(n)$ is a bandwidth parameter.  Explicitly
\begin{equation}\label{eq:GPHexplicit}
\hat d_{m} = s_m^{-2} \sum_{k=1}^m \nu_k \log(I_{2k+1}),
\end{equation}
with $\nu_k = -2 \Bigl( \log |2\sin(\lambda_{2k+1}/2)| - \frac{1}{m}\sum_{j=1}^m
\log |2\sin(\lambda_{j}/2)| \Bigr)$ and $s_m^2 = \sum_{k=1}^m \nu_k^2$.  We
consider $\esp[(\hat d_m -d)^2]$ the mean square error (MSE) of the GPH
estimator.  Theorem \ref{theo:rateoptimality} gives a bound on the MSE which is
uniform over a class of long-range dependent linear processes, from which rate
optimality can be deduced.  

\begin{theo}\label{theo:rateoptimality}
  Under the assumptions of Theorem~\ref{theo:edgdftlong} with $s \geq 5$ and
  conditions (\ref{cond:local}), there exists a constant $C$ which depends only on
  $\beta$, $\delta$, $\Delta$, $\vartheta$, $\mu$ and the distribution of $Z_1$ such that
\[
\esp[(\hat d_m -d)^2] \leq C \left\{ \left(\frac mn \right)^{2\beta} + \frac1m
  \right\}.
\]
  With $m$ proportional to $n^{2\beta/ (2 \beta +1)}$, $\esp[(\hat d_m -d)^2] \leq C n^{-2\beta/(2\beta+1)}$.  
\end{theo}

\begin{rem}
  The condition $s \geq 5$ seems slightly stronger than necessary for bounding
  the MSE of $\hat d$. But it is allows the function
  $h(x_1,\dots,x_4) = g(x_1,x_2)g(x_3,x_4)$ with $g(\bx) = \log(\|\bx\|^2)
  -\bar\eta$ to have finite $N_s(h)$ norm (see
  Corollary~\ref{coro:boundmomshort2} and the remark that follows.
\end{rem}

\subsubsection{Monte Carlo results}
\label{sec:monte-carlo-results}

Assuming more stringent global condition on the regularity of the spectral
density allows one to evaluate the bias term in the decomposition of
the mean squared error.  For comparison, using the specific set of assumptions
\citet*{hurvich:deo:brodsky:1998}, we can prove that the leading terms in the
MSE are of the form $ am^4/n^4 + b/m $ for bandwidth such that $\lim_{n
  \rightarrow \infty} {1}/{m} + {m\log(m)}/{n} = 0$.  The constant $a$ and $b$
can be made uniform in the class of spectral densities under consideration.  It
shows that the MSE of the GPH estimator is asymptotically insensitive to the
distribution of the innovation as soon as this distribution satisfies some
moment and regularity conditions. Finite sample implications of this statement
is illustrated here by the results of a Monte Carlo study. For sample sizes
$n=250,500,1000,2500,5000$, we have simulated 1000 realizations of a
FARIMA$(1,d,0)$ processes defined by
  \begin{align*}
    (1-B)^{0.3}(1-0.3B)X_t = Z_t
  \end{align*}
  where $B$ is the back-shift operator and $(Z_t)_{t \in \Zset}$ is a zero mean
  unit variance i.i.d sequence with the following marginal distributions (a)
  Gaussian (b) Laplacian (c) zero-mean (shifted) Pareto, with
\begin{align*}
  \prob( Z_0 \leq u) = (1 - (u+7/6)^{-7}) \ind{u \geq -7/6}.
  \end{align*}
  Whereas it is possible to simulate exactly a Gaussian FARIMA$(p,d,q)$ process
  (\textit{e.g.} computing the covariance structure and using Levinson-Durbin
  algorithm), there is no general way to do it for non Gaussian processes.  In
  the Monte-Carlo experiment, the process $(X_t)$ is obtained using a truncated
  MA($\infty$) representation.
  For each realization of each process, we evaluate the squared error $(\hat
  d_m - d)$ and define the Monte Carlo MSE as the average of those errors.
  We have focused on the sensitivity with respect to the distribution of $Z$ of
  the bandwidth $m$ which is optimal in the MSE sense. Figure~\ref{fig:mseplot}
  and Table~\ref{tab:results} show that for sample size $n=250$ the MSE is
  minimized at $m=37$ or $38$ which means that the optimal bandwidth is about
  the same for those three linear processes. Figure~\ref{fig:gphboxplot}
  represents the box-and-whiskers plot of the GPH estimator for two different
  sample sizes and the three models we are concerned with. Here again, the
  sensitivity with respect to the distribution of the driving noise is hardly
  discernible.  In Table~\ref{tab:results} we displayed the value of the bias
  and of the mean square error of the GPH at this estimated optimal bandwidth.

\begin{figure}[htbp]
  \centering \includegraphics[width=6cm]{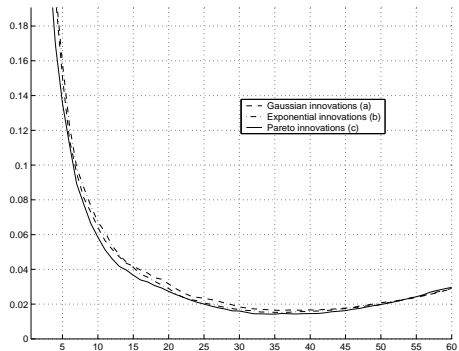}
  \caption{Comparisons of the MSE versus the bandwidth for the FARIMA processes
  (a),(b) and (c). Sample size $n = 250$}
  \label{fig:mseplot}
\end{figure}
\begin{figure}[htbp]
  \centering \includegraphics[width=6cm]{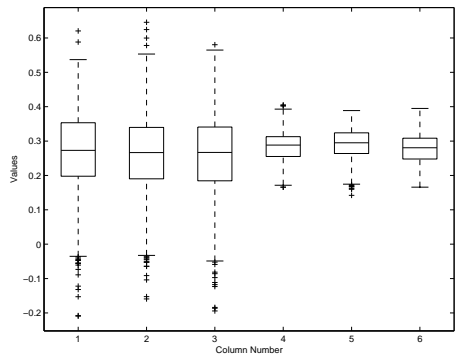}
  \caption{Box-plot of the GPH estimator for processes
  (a),(b) and (c), sample size $n = 250, 2500$}
  \label{fig:gphboxplot}
\end{figure}

{\scriptsize
\begin{table}[htbp]
  \centering
  \begin{tabular}{ccccc}

 &&(a)&(b)&(c)\\ \\ \multirow{3}{2cm}{$n$=250} &$m_{opt}$ & 37 & 37 & 38 \\ 
& $\esp_{\text{MC}}( \hat d_{m_{opt}} - d)$  & -0.03037 & -0.03751 & -0.03871 \\ 
& $\esp_{\text{MC}}(( \hat d_{m_{opt}} - d) ^2 )$  & 0.01507 & 0.01508 & 0.01569 \\ 
\\\multirow{3}{2cm}{$n$=500} &$m_{opt}$ & 64 & 60 & 61 \\ 
& $\esp_{\text{MC}}( \hat d_{m_{opt}} - d)$  & -0.02813 & -0.01871 & -0.02778 \\ 
& $\esp_{\text{MC}}(( \hat d_{m_{opt}} - d) ^2 )$  & 0.00807 & 0.00734 & 0.00766 \\ 
\\\multirow{3}{2cm}{$n$=1000} &$m_{opt}$ & 106 & 117 & 107 \\ 
& $\esp_{\text{MC}}( \hat d_{m_{opt}} - d)$  & -0.01984 & -0.02377 & -0.01737 \\ 
& $\esp_{\text{MC}}(( \hat d_{m_{opt}} - d) ^2 )$  & 0.00507 & 0.00393 & 0.00438 \\ 
\\\multirow{3}{2cm}{$n$=2500} &$m_{opt}$ & 222 & 212 & 238 \\ 
& $\esp_{\text{MC}}( \hat d_{m_{opt}} - d)$  & -0.01492 & -0.00967 & -0.02037 \\ 
& $\esp_{\text{MC}}(( \hat d_{m_{opt}} - d) ^2 )$  & 0.00207 & 0.00202 & 0.00219 \\ 
\\\multirow{3}{2cm}{$n$=5000} &$m_{opt}$ & 377 & 385 & 370 \\ 
& $\esp_{\text{MC}}( \hat d_{m_{opt}} - d)$  & -0.01097 & -0.00937 & -0.01269 \\ 
& $\esp_{\text{MC}}(( \hat d_{m_{opt}} - d) ^2 )$  & 0.00106 & 0.00104 & 0.00097 \\ 
\\
  \end{tabular}
  \caption{Optimal bandwidth, bias and MSE for processes (a), (b) and (c) and
    different sample sizes $n$. All those values are estimated by Monte Carlo}
  \label{tab:results}
\end{table}
}

\appendix

\section{Edgeworth expansion for triangular arrays}
\label{sec:edegw-expans-triang}
In this section we recall the theorem established in
\citet*{fay:moulines:soulier:2004}.  Let $( Z_t )_{t \in \Zset}$ be an i.i.d
sequence and $(\bU_{n,j})_{j\in \Zset, n\in \Nset}$ an array of vectors in
$\Rset^u$,  where $u$ is an integer. Define $\bS_n = \sum_{j\in \Zset}
\bU_{n,j} Z_j$ and let $\bV_n = \sum_{j\in\Zset} \bU_{n,j}\bU'_{n,j}$.  For
$\bnu\in\Nset^{u}$, $2 \leq |\bnu| \leq s$, denote $\chi_{n,\bnu}$ the
cumulants of $\bS_n$. Then $ \chi_{n,\bnu} = \kappa_{|\bnu|} \sum_{j \in \Zset}
\bU_{n,j}^{\bnu}.  $ where $\kappa_r$ denotes the $r$-th cumulant of $Z_1$, $r
\leq s$.  Consider the following assumptions.
\begin{fms}{edg:normalize-Vn}
  There exist positive constants $v_*$ and $v^*$ such that
\begin{align*}
  v_* \leq \liminf_n v_{\min}[\bV_n] \leq \limsup_n v_{\max}[\bV_n] \leq v^*
\end{align*}
where $v_{\min}[\bV_n]$ (resp. $v_{\max}[\bV_n]$) is the smallest (resp. the
largest) eigenvalue of $\bV_n$.
\end{fms}
\begin{fms}{edg:lindsmall}
  There exist positive constants $\eta$, $c_0$, a sequence $(M_n)_{n \in
    \Nset}$ of positive numbers, and a sequence $(J_n)_{n \in \Nset}$ of
    subsets of $\Zset$, such that, for all $n \geq 0$
\begin{gather}
\label{eq:edg:majorizingsequence}
\sup_{j \in \Zset} \|\bU_{n,j}\| \leq M_n \\
\label{eq:edg:lindeberg}
\lim_{n\to\infty} M_n = 0 \\
\label{eq:edg:smallness}
\mathrm{card}(J_n) \leq c_0 M_n^{-2} \quad \text{and} \quad \frac{\sum_{j \in
J_n} \|\bU_{n,j}\|^2}{\sum_{j \in \Zset} \|\bU_{n,j}\|^2} \geq \eta.
\end{gather}
\end{fms}
\begin{fms}{edg:l1}
There exist $\zeta \geq1$ and a sequence $(M_n)_{n \in \Nset}$
satisfying~\eqref{eq:edg:majorizingsequence} such that
\[  \sup_{n\geq 0} M_n^{\zeta} \sum_{j\in\Zset} \|\bU_{n,j}\| < \infty \; . \]
\end{fms}
\begin{theo}[\citealp*{fay:moulines:soulier:2004}]
  \label{theo:devedgta}
  Let $s\geq3$, and $p'\geq0$ be integers and $p\geq1$ be a real number. Assume
  \refhyp{Lp}($s,p,p'$), \reffms{edg:normalize-Vn} and \reffms{edg:lindsmall}.
  Assume in addition either \reffms{edg:l1} or $p' \geq s$ in
  \refhyp{Lp}($s,p,p'$).
  Then, there exist a constant $C$ and an integer $N$ (depending only on the
  distribution of $Z_1$, and the constants appearing in the assumptions) such
  that, for all $n \geq N$, the distribution of $\bS_n$ has a density $q_n$
  with respect to Lebesgue measure on $\Rset^u$ which satisfies
  \begin{align}
    \sup_{\bx \in \Rset^u} (1 + \|\bx\|^s)\Bigl|q_n(\bx) - \sum_{r=0}^{s-3}
    P_r(\bx,\bV_n,\{\chi_{n,\bnu}\}) \Bigr| \leq C \sum_{j \in \Zset}
    \|\bU_{n,j}\|^s
    \label{eq:bounddensity}
  \end{align}
\end{theo}

\section{Proof of Theorem \ref{theo:edgdftshort}}
\label{sec:proof-theor-edgshort}
  The proof consists in checking that assumptions \reffms{edg:normalize-Vn},
  \reffms{edg:lindsmall} and \reffms{edg:l1} hold uniformly with respect
  to $\psi \in \mathcal G(\alpha,\beta,\delta)$ and $\bk$ for $\bU_{n,j}$'s of the
  form~(\ref{eq:defUnj}). 
  To prove \reffms{edg:normalize-Vn}, write $\bV_n(\bk) = \bV(\bk) +
  \bW_n(\bk)$, with $\bV(\bk)$ defined in~(\ref{eq:defbvbk}).  Define $\|\bW\|_1 = \max_{1\leq i \leq v}\sum_{j=1}^v
  |w_{i,j}|$ for any matrix $\bW=(w_{i,j})_{1\leq i,j \leq v}$.  Similarly to
  \citet[p. 54]{hannan:1960}, we have under (\ref{eq:weakdependence})
\begin{equation}
\label{eq:boundnormwnk}
\| \bW_n(\bk) \|_1 \leq C(\alpha,\beta) n^{-1}.
\end{equation}
The matrices $\bV(\bk)$ have the following algebraic property.
\begin{lem}
\label{lem:algebrique}
 There exist two positive constants $v_*$ and $v^*$ such that
\begin{equation}
  \label{eq:pauleta}
  2 v_* \leq \inf v_{\min} [\bV(\bk)] \leq \sup v_{\max}[\bV(\bk)]  \leq 2 v^*
\end{equation}
where the infimum and supremum are taken over all the $u$-tuples of distinct integers in $\Nset^u$.
\end{lem}
\begin{proof}
Noting that $\text{trace}[\bV(\bk)] = u$,
\begin{align}
\label{eq:boundvmax}
v_{\max}[\bV(\bk)] \leq \text{trace}[\bV(\bk)]/{2u} = 1/2.
\end{align}
Take $v^* = 1/4$.
Recall  that $k_1 < \dots < k_u$. Note  that, for any $n \geq 2k_u+2r+1$, $\bV(\bk)$ is the covariance matrix of
$$\sqrt{2\pi}(c^Y_{r,n,k_1}, s^Y_{r,n,k_1}, \dotsc, c^Y_{r,n,k_u}, s^Y_{r,n,k_u})$$ with
$c^Y_{r,n,k} = (2 \pi a_r n)^{-1/2} \sum_{t=1}^n h_{t,n}^r Y_t \cos( t
\lambda_k ) $ and $s^Y_{r,n,k} = (2 \pi a_r n)^{-1/2} \sum_{t=1}^n h_{t,n}^r
Y_t \sin( t \lambda_k ) $ the sine and cosine transform of a unit-variance
zero-mean Gaussian white noise $(Y_n)_{n \in \Zset}$. Recall that
\begin{equation}
\label{eq:mbami}
c^Y_{r,n,k} = a_r^{-1/2}\sum_{l=0}^r (-1)^l \left(^r_l\right) \; c^Y_{0,n,k+l}\
\ \ \text{and} \ \ \ s^Y_{r,n,k} = a_r^{-1/2}\sum_{l=0}^r (-1)^l
\left(^r_l\right) \; s^Y_{0,n,k+l} \; .
\end{equation}
The random variables $c_{0,n,k}$ and $s_{0,n,k}$, $k=1,\dots,[(n-1)/2]$ are
centered i.i.d Gaussian with variance $1/4\pi$. Assume that $\bV(\bk)$ is
not invertible. It yields that for some $2u$-tuple of reals
$(\alpha_1,\beta_1,\cdots,\alpha_u,\beta_u) \neq(0,0,\cdots,0,0)$,
\begin{equation*}
  \sum_{j=1}^u (\alpha_j c_{r,n,k_j} + \beta_j s_{r,n,k_j})
  \stackrel{\mathbb{L}^2}{=}0 .
\end{equation*}
 Then by (\ref{eq:mbami}), there exists a
linear combination of $c_{0,n,k}$'s and $s_{0,n,k}$'s that is equal to zero. 
$c_{0,n,k_{u}+r}$ and $s_{0,n,k_{u}+r}$ appear in this combination with coefficients
$a_r^{-1/2}(-1)^r \alpha_u$ and $a_r^{-1/2}(-1)^r \beta_u$, respectively. It follows from the
independence and non-degeneracy of the $c_{0,n,k}$'s and $s_{0,n,k}$'s that $\alpha_u = \beta_u = 0$.
Iterating the argument yields the contradiction
$\alpha_u=\beta_u=\alpha_{u-1}=\beta_{u-1}=\cdots=\alpha_1=\beta_1=0$. 
Thus for any
$u$-tuple $\bk$ of distinct integers
\begin{align}
\label{eq:vminunefreq}
v_{\min}[\bV(\bk)] > 0 .
\end{align}
It remains to prove that $v_{\min}[\bV(\bk)]$ is bounded away from zero
uniformly in $\bk$. Define
$$K_u = \{ \bk = (k_1,\cdots,k_{u'}) \in \Nset^{u'} ,1 \leq u' \leq u, 0 <
k_{i+1} - k_i \leq r\} .$$ Note now that by (\ref{eq:defbvbk}),
$v_{\min}[\bV(\bk)]$ is a function of the vector $(k_2 - k_1, k_3 - k_2 , \cdots , k_{u}
- k_{u-1})$ thus taking finitely many different values on $K_u$. From the this remark and (\ref{eq:vminunefreq}),
 \begin{equation}
   \label{eq:vminfiniteset}
   v_1 \eqdef \inf_{\bk \in K_u} v_{\min}[\bV(\bk)] > 0
 \end{equation}
 since the infimum is taken on a finite set of positive values. Consider now a
 $u$-tuple $\bk$ that does not belong to $K_u$; In this case, for some $i \in
 \{1,\cdots,u-1\}$, $k_{i+1} - k_i > r$, and then $\bk$ may be partitioned as
 $L \geq 2$ blocks of indexes $(\bk_1,\dots,\bk_L)$ such that all the $\bk_i$'s
 belong to $K_u$ and, for all $i\in\{1, \cdots, L-1\}$, $\min \bk_{i+1} - \max
 \bk_{i} > r$. Let $l_i$ denotes the length of the block $\bk_i,
 i=1,\cdots,L$. By this construction and (\ref{eq:defbvbk}), the matrix
 $\bV(\bk)$ has a block-diagonal structure
\begin{align*}
  \bV(\bk) = \left(
    \begin{array}{ccc}
      \bV(\bk_1) && \boldsymbol{0} \\ & \ddots &\\ \boldsymbol{0} & &
       \bV(\bk_L)
    \end{array}
\right) \; .
\end{align*}
Using (\ref{eq:vminfiniteset}),
\begin{equation}
  \label{eq:bounddetvn}
v_{\min}[\bV(\bk)] (v_{\max}[\bV(\bk)])^{2u-1} \geq \det [\bV(\bk)] = \prod_{i =
1}^L \det[\bV(\bk_i)] \geq \prod_{i = 1}^L v_1^{2l_i} = v_1^{2u}.
\end{equation}
We conclude from (\ref{eq:bounddetvn}) and (\ref{eq:boundvmax}) that
\begin{equation}
  \label{eq:vmingeneral}
  v_{\min}[\bV(\bk)] \geq v_1^{2u} 2^{2u-1} =: v_2 > 0.
\end{equation}
(\ref{eq:pauleta}) follows from (\ref{eq:vminfiniteset}) and
(\ref{eq:vmingeneral}) with $v_* = \frac12 \min(v_1,v_2)$.
\end{proof}
\begin{proof}[Proof of Theorem~\ref{theo:edgdftshort}]
  By (\ref{eq:boundnormwnk}) and Lemma~\ref{lem:algebrique},
  \reffms{edg:normalize-Vn} holds with $v_*$ and $v^*$ of
  Lemma~\ref{lem:algebrique}, for some $N_0$, $n \geq N_0$ and uniformly in
  $\bk,\alpha$ and $\beta$. With~(\ref{eq:boundvmax}),
\begin{align}
  \Bigl| \sum_{j \in \Zset} \|\bU_{n,j}(\bk)\|^2 - u \Bigr| =
  \left|\text{trace}[\bV_n(\bk)] - u \right| \leq  C(\alpha,\beta)n^{-1/2}.\label{eq:1}
\end{align}
Prove now that \reffms{edg:lindsmall} is verified.  Since $f$ is bounded
away from zero and $\sum\nolimits_{j\in\Zset}|\psi_{j}| \leq \beta < \infty$,
(\ref{eq:edg:majorizingsequence}) and (\ref{eq:edg:lindeberg}) are verified
with $M_n \eqdef C(r)\beta \alpha^{-1/2} n^{-1/2}$.
Put $J_n = \{j , |j| < 2n\}$. Then $\text{card} (J_n) \leq c_0
M_n^{-2}$ for some $c_0$ depending only on $r,\alpha,\beta$ and
\begin{multline}
\label{eq:2}
  \sum_{|j| \in \Zset \setminus J_n} \|\bU_{n,j}(\bk) \|^2 \leq C(r)\alpha^{-2}
  n^{-1} \sum_{|j| \geq 2n} \left(\sum_{t=1}^{n} |\psi_{t+j}|\right)^2
  \\\leq C(r)\alpha^{-2} \sum_{|j| \geq 2 n}\sum_{t=1}^{n}
  \psi_{t+j}^2 \leq C(r)\alpha^{-2} \sum_{|j| \geq n} |j|\psi_j^2.
\end{multline}
Under (\ref{eq:weakdependence}), $|\psi_j| \leq \beta|j|^{-1/2-\delta}$ so that
\begin{align}
  \label{eq:4}
  \sum_{|j| \geq n} |j|\psi_j^2 \leq \beta n^{-2\delta} \sum_{|j| \geq n}
  |j|^{1/2+\delta}|\psi_j| \leq \beta^2 n^{-2\delta}
\end{align}
For any $\epsilon >0$ and large enough $n$, $ \sum_{|j| \geq n}
\|\bU_{n,j}(\bk) \|^2 \leq \epsilon$ uniformly in $\bk$ and $\psi \in \mathcal
G(\alpha,\beta,\delta)$. \eq{edg:smallness} follows from
(\ref{eq:1}), (\ref{eq:2}) and (\ref{eq:4}).  Finally,
\begin{multline*}
\sum_{j \in \Zset} \|\bU_{n,j}(\bk)\| \leq C(r) \alpha^{-1/2} n^{-1/2}
\sum_{t=1}^n \sum_{j\in\Zset} |\psi_{t+j}| = C(r) \alpha^{-1/2} n^{1/2} \sum_{j
\in \Zset} |\psi_j| \\ \leq C(r) \alpha^{-1/2} \beta n^{1/2} = C(r)^2
\alpha^{-1}\beta^2 M_n^{-1}
\end{multline*}
so that \reffms{edg:l1} holds with $\zeta=1$. \end{proof}

\section{Proof of Lemma~\ref{lem:covdft}}
The proof is an adaptation of \cite{lang:soulier:2002} to fit our need of
uniformity of the bounds with respect to the function $\psi$ whether it belongs
to $\Fglo$ or $\Floc$ only. For sake of brevity, the proof is omitted and we
refer the interested reader to their paper. It derives from their more general
analytical lemma that we recall here.
\begin{lem}[\cite{lang:soulier:2002}] \label{lem:bornes}
  Let $q\in\Nset$, $K\geq 1$, $\vartheta\in(0,\pi]$.  Let $\psi$ be an
  integrable function on $[-\pi,\pi]$, such that for all
  $x\in(0,\vartheta]\setminus\{0\}$, $\psi(-x) = \bar \psi(x)$ and
\begin{gather} \label{eq:derivee}
  |\psi(x)-\psi(y)| \leq K \frac{|\psi(x)| + |\psi(y)|}{x \wedge y} \, |x-y|,
  \quad \text{for all} (x,y) \in (0,\vartheta] \times (0,\vartheta]
\end{gather}
Assume that $|\psi|$ is regularly varying at zero with index $\rho$ and that
\[
  \psi(x) = x^\rho c(x) \exp \left\{ \int_x^{\vartheta} \frac{g(s)}s \d s \right\}
\]
with  (i) $\lim_{x\to0}g(x)=0$; (ii) $\lim_{x\to0}c(x)$ exists in $(0,\infty)$.  Let
$D_{n}$ be such that for $x \in [-\pi,\pi]$,
\begin{gather} \label{eq:noyau}
  |D_{n}(x)| \leq C \frac{n^{1/2}}{(1+n|x|)^{q+1}}.
\end{gather}
\begin{itemize}
\item 
If $\rho\in(-1,2q+1)$, there exists a constant $C$ such that, for all $n \geq
1$ and all $k$ such that $0 < x_k \leq \vartheta/2$,
\begin{gather} \label{eq:borneqk}
  \left | \int_{-\pi}^{\pi} \left( \frac{\psi(x)}{\psi(x_k)} - 1 \right)
    |D_n(x_k -x)|^2 \d x \right | \leq C \log^{\nu(q)}(k) k^{-1}.
\end{gather}
with $\nu(0)=1$ and $\nu(q)=0$ if $q\geq1$.
\item
If $\rho \in (-1/2,q+1/2)$, there exists a constant $C$ such that, for all $n
\geq 1$ and all integers $k,j$ such that $0 < x_k \ne x_j \leq \vartheta/2$,
\begin{multline}  \label{eq:borneqjk}
  \left| \int_{-\pi}^{\pi} \left( \frac{\psi(x)}{\psi(x_k)}-1 \right) D_n(x_k
    -x) \overline{D_n(x_j-x)} \d x \right | \\ + \left| \int_{-\pi}^{\pi}
    \left( \frac{\psi(x)}{\psi(x_k)}-1 \right) D_n(x_k -x) D_n(x+x_j) \d x
    \right | \\ \leq \left\{ \begin{array}{ll} C (1+|\psi(x_j)/\psi(_k)|)
    |k-j|^{-q} (j \vee k)^{-1} \leq C (jk)^{-1/2} & (q>0); \\ C (jk)^{-1/2}
    \log(j \vee k) & (q=0).
\end{array} \right.
\end{multline}
\item For any $\beta>0$, if $\rho\in(-1,2q+1)$, for any integer $k$ such that $0 <
x_k \leq \vartheta/2$,
\begin{gather}  \label{eq:bornebeta1}
  \left| \int_{-\pi}^{\pi} \frac{\psi(x) |x|^\beta }{\psi(x_k)} |D_n(x_k -x)|^2
  \d x \right | \leq C \left\{ k^{-2q-1} + (k/n)^\beta \right\}.
\end{gather}
\item If $\rho\in(-1/2,q+1/2)$, for any integers $j,k$ such that $0 < x_k \ne x_j
\leq \vartheta/2$,
\begin{multline} \label{eq:bornebeta} \left| \int_{-\pi}^{\pi} \frac{\psi(x)
      |x|^\beta }{\psi(x_k)} |D_n(x_k -x) \overline{D_n(x_j-x)}| \d x \right |
  + \left| \int_{-\pi}^{\pi} \frac{\psi(x) |x|^\beta }{\psi(x_k)} |D_n(x_k -x)
    D_n(x_j+x)| \d x \right | \\ \leq C (1 + |\psi(x_j)/\psi(x_k)|) \left\{
    (jk)^{-q} (j\vee k)^{-1} + |j-k|^{-q-1}((j \vee k)/n)^\beta \log^{\nu(q)}(j
    \vee k) \right\} \\ \leq C \left\{ (jk)^{-1/2} + ((j \vee k)/n)^\beta
    \log^{\nu(q)}(j \vee k) \right\}.
\end{multline}
\end{itemize}
\end{lem}

\section{Proof of Theorem \ref{theo:edgdftlong}}
\label{sec:proof-theor-edglong}
The proof of Theorem~\ref{theo:edgdftlong} consists in checking that
assumptions \reffms{edg:normalize-Vn}, \reffms{edg:lindsmall} hold uniformly.
\begin{lem}
\label{lem:algebrique2}
There exist integers $N_0$, $K_0$, and $v_* > 0, v^*>0$ (depending only on
$\vartheta,\beta,\delta,\Delta,\mu$) such that, for all $n \geq N_0$, we have,
\begin{enumerate}
\item for all $u$-tuple $\bk$ of distinct integers, $1 \leq \min \bk \leq
\max \bk \leq m_n$,
\begin{equation}
\label{eq:bornesup}
 v_{\max}[\bV_n(\bk)] \leq v^*  \; ;
\end{equation}
\item for all integer $k$, $1 \leq k \leq m_n$
\begin{equation}
\label{eq:borneinfsinglefreq}
v_* \leq v_{\min}[\bV_n(k)] \; ;
\end{equation}
\item for all $u$-tuple $\bk$ of distinct integers, $K_0 \leq \min \bk \leq
\max \bk \leq m_n$,
\begin{equation}
\label{eq:borneinfcov}
v_* \leq v_{\min} [\bV_n(\bk)]  \; .
\end{equation}
\end{enumerate}
\end{lem}
\begin{proof}
As in Appendix~\ref{sec:proof-theor-edgshort}, we put $\bV_n(\bk) = \bV(\bk) +
\bW_n(\bk)$ where $\bV(\bk)$ is defined in~(\ref{eq:defbvbk}).
Applying Lemma~\ref{lem:covdft}, we obtain
\begin{equation}\label{eq:radius}
\|\bW_{n}(\bk)\|_1 \leq
C
(\vartheta,\beta,\delta,\Delta,\mu)
\begin{cases}
\frac{1}{k_1} + \left( \frac{m_n}{n}\right)^\beta \log \left( \frac{m_n}{n}
\right) & \text{under} \quad (\ref{cond:local}),\\ \frac{1}{k_1} &
\text{under} \quad (\ref{cond:global}).
\end{cases}
\end{equation}
(\ref{eq:bornesup}) follows immediately. The proof of \eqref{eq:borneinfcov}
follows by picking $N_0, K_0$ large enough. For \eqref{eq:borneinfsinglefreq},
it remains to prove that for any integer $k$, $1 \leq k \leq K_0$, $\bV_n(k)$
converges to a positive definite matrix $\widetilde\bV(k)$ and that this convergence is
uniform w.r.t to $\psi$, for $\psi \in
\Floc(\vartheta,\beta,\delta,\Delta,\mu)$ or $\psi \in
\Fglo(\vartheta,\beta,\delta,\Delta,\mu)$.  What follows is an adaptation of
\citep[Lemma 7.3]{iouditsky:moulines:soulier:2001}. Write
\begin{equation}
   \esp [ |\omega_{r,n,k} |^2] =
  \frac1{f(\lambda_k)} \left( \int_{|\lambda|\leq \vartheta\pi} +
    \int_{|\lambda| > \vartheta \pi} \right)
  | D_{r,n}(\lambda-\lambda_k) |^2 f(\lambda) d
  \lambda  =: A_1 + A_2 \label{eq:decompmom}
\end{equation}
where $D_{r,n}$ is defined in~(\ref{eq:form:kernel}).
For $n \geq 4\pi K_0/\vartheta$, $1 \leq k \leq K_0$ and $|\lambda|\geq
\vartheta\pi$, $|n(\lambda - \lambda_k)| \geq
n\vartheta / 2$. Using~(\ref{eq:decaytaper}) and~(\ref{eq:cond1bis}), we get
\begin{equation}
  A_2 \leq
  \frac{C\lambda_k^{2d}}{\lambda_k^{2d}f(\lambda_k)}n^{-2r-1}\int_{|\lambda| >
    \vartheta \pi} \lambda^{2d}f(\lambda) d\lambda \leq Cn^{-2r} \label{eq:bornegrandlambda}
\end{equation}
By change of variable,   $$ A_1  = \frac{n^{2d}
    |1-\emi{\lambda_k}|^{2d}}{f^*(\lambda_k)}
   \int_{|\lambda| \leq n\vartheta} \left| n^{-1/2} D_{r,n}(\lambda /  n
     - \lambda_k)
   \right|^2 n^{-2d} |1-\emi{
    \lambda/n}|^{-2d}f^*(\frac \lambda n) d \lambda.
$$
Write $
 \lim_{n\to\infty} n^{-1/2} D_{r,n} (\lambda / n ) = \frac{1}{\sqrt{2 \pi a_r}} \int_0^1
(1-\expe^{2\rmi \pi s})^r \emi{s\lambda} ds =: \hat h_r(\lambda)$.
By Riemann approximation, it can be seen that $|n^{-1/2} D_{r,n} (\lambda/ n) -
\hat h_r(\lambda) | \leq C(1+|\lambda|)/n$. Note also that  $|\hat
h_r(\lambda)|\leq C|\lambda|^{-r-1}$. Then
\begin{multline}
  \label{eq:approxriemann}
  \Biggl| A_1 - \frac{n^{2d}
    |1-\ei{\lambda_k}|^{2d}}{f^*(\lambda_k)}
  \times \int_{-n\vartheta}^{n\vartheta} \left| \hat h_r (\lambda - 2 \pi k ) \right|^2
 n^{-2d} |1-\ei{\lambda/n}|^{-2d}f^*(\frac \lambda n) d\lambda
\Biggr| \\\leq C k^{2d} n^{-r} \leq C n^{-r}.
\end{multline}
Here and in the following, $C$ is a generic constant which depends only on
$\vartheta,\beta,\delta,\Delta,\mu,r$ and $K_0$.
For $|\lambda| \leq n\vartheta$, using~(\ref{cond:local}),
\begin{align}
& \frac{f^*(0)}{f^*(\lambda_k)}    \left| n^{-2d} |1-\ei{\lambda/n}|^{-2d} - |\lambda|^{-2d}
  \right| +
\frac{ | f^*(\frac \lambda n) - f^*(0) |}{f^*(\lambda_k)}  n^{-2d} |1-\ei{\lambda/n}|^{-2d}
\nonumber \\
&\leq C  \left\{ \left| n^{-2d} |1-\ei{\lambda/n}|^{-2d} - |\lambda|^{-2d}
  \right| + n^{-2d} |1-\ei{\lambda/n}|^{-2d} |\frac \lambda n |^{\beta'}
\right\}
\label{eq:bound2}
\end{align}
with $\beta' = \beta \wedge 1$.
For $ x \in [-\pi,\pi], \frac 2 \pi |x| \leq |\ei{x} -1| = |2 \sin \frac x 2|
\leq  |x|$ and $ ||\ei{x} -1| - |x|| \leq \ x^2/2$. Also,  for
any $\upsilon \in \Rset, x>0,y>0, |x^\upsilon - y^\upsilon| \leq |\upsilon|
(x^{\upsilon - 1} \vee y^{\upsilon - 1}) |x-y|$. Using those relations,
write, for $\lambda \in [-n\pi , n\pi]$,
\begin{align*}
   \Bigl| n^{-2d} |1-\ei{\lambda/n}|^{-2d} - |\lambda|^{-2d} \Bigr|
& \leq  C n^{-2d} |\frac{\lambda}{n}|^{-2d-1}\Bigl||1 -\ei{\lambda / n}| - |\frac \lambda n| \Bigr|\\
& \leq  C n^{-1} |\lambda|^{-2d+1}
\end{align*}
Then
\begin{align}
 \int_{-n\vartheta}^{n\vartheta}  &
  |\hat h_r (2 \pi k - \lambda)|^2  \left| n^{-2d} |1-\ei{\lambda/n}|^{-2d} - |\lambda|^{-2d}
  \right| d\lambda \nonumber \\
&\leq   C n^{-1}  \int_{-n\vartheta}^{n\vartheta}    |\hat h_r (2 \pi k - \lambda)|^2 |\lambda|^{-2d+1} d\lambda  \nonumber \\
&\leq   C n^{-1} \int_{-n\vartheta}^{n\vartheta}   |\hat h_r (2 \pi k - \lambda)|^2 |\lambda|^{2r} d\lambda  \leq C n^{-1} \label{eq:vahid1}
\end{align}
and
\begin{align}
 \int_{-n\vartheta}^{n\vartheta}  &
|\hat h_r (2 \pi k - \lambda)|^2 n^{-2d}
  |1-\ei{\lambda/n}|^{-2d} |\frac \lambda n |^{\beta'} d\lambda \nonumber  \\
  &\leq \
n^{-\beta'} \int_{-n\vartheta}^{n\vartheta}
  |\hat h_r (2 \pi k - \lambda)|^2 |\lambda|^{-2d
    + \beta'} d\lambda \nonumber  \\
  &\leq C n^{-\beta'}.
  \label{eq:vahid2}
\end{align}
Gathering~(\ref{eq:decompmom}),~(\ref{eq:bornegrandlambda}),~(\ref{eq:approxriemann}),~(\ref{eq:bound2}),~(\ref{eq:vahid1}),~(\ref{eq:vahid2}) yields
\begin{equation*}
  \Biggl| \esp [|\omega_{r,n,k}|^2] - \frac{(2\pi)^{2d} k^{2d}
    f^*(0)}{f^*(\lambda_k)} \int_{\infty}^{+\infty} \left| \hat
  h_r (\lambda - 2 \pi k ) \right|^2  |\lambda|^{-2d} d\lambda
 \Biggr| \leq Cn^{-\beta'}
\end{equation*}
Similar arguments leads to 
\begin{equation*}
  \Biggl| \esp [\omega_{r,n,k}^2] - \frac{(2\pi)^{2d} k^{2d}
    f^*(0)}{f^*(\lambda_k)} \int_{\infty}^{+\infty} \hat
  h_r (\lambda - 2 \pi k ) \hat
  h_r (\lambda + 2 \pi k )   |\lambda|^{-2d} d\lambda
  \Biggr| \leq Cn^{-\beta'}
\end{equation*}
Defining the scalar product $(u,v)_d = \int_\Rset
u(\lambda)v(\lambda)|\lambda|^{-2d} d\lambda$, Then $\det \bV_n(k)$ is
uniformly approximated by the Gram determinant of the functions $\hat
h_r(\lambda - 2k\pi)$ and $\hat h_r(\lambda + 2k\pi)$ associated with the
product $(\cdot,\cdot)_d$ and then is a continuous function of $\eta_k(d) :=
\limn \esp [\left| \omega_{r,n,k} \right| ^2]$ and $\eta'_k(d) := \limn \esp [
\omega_{r,n,k}^2]$.. The whole set of functions $\hat h_r(\lambda + 2j\pi),
j\in \Zset$ is linearly independent, so that those determinant are positive.
Using continuity of $\eta_k$ and $\eta'_k$ w.r.t. $d$, the infimum on the
compact set $[-\Delta,\delta]$ and the minimum over $k=1,\dots,K_0$ is positive
too, which concludes the proof.
\end{proof}
\begin{lem}
\label{lem:spectral}
There exists a constant $C$ (depending only on $\vartheta$, $\beta$, $\delta$,
$\Delta$, $\mu$,$r$) such that for all $k \in \{1,\dots,\tilde n\}$,
\begin{equation}
  \label{eq:spectral}
 \frac1{\sqrt{n f(\lambda_k)}} \left| \sum_{t=1}^n h_{t,n}^r \psi_{t+j}
 \expe^{\rmi t \lambda_k} \right| \leq C n^{-1/2}.
\end{equation}
\end{lem}
\begin{proof}
 The main tool of the proof is the bound~(\ref{eq:decaytaper}) and the
 technique are the same as the one used in the proof of Lemma~\ref{lem:covdft}.
Decompose
\begin{align*}
   |\psi(\lambda_k)|^{-1} \frac{1}{\sqrt{2 \pi a_r n}} \sum_{t=1}^n h_{t,n}^r
  \psi_{t-j} \expe^{\rmi t \lambda_k} = |\psi(\lambda_k)|^{-1} \int_{-\pi}^\pi
  \psi(\lambda) \expe^{\rmi j \lambda} D_{n,r}(\lambda_k - \lambda) d\lambda 
\end{align*}
into
\begin{align*}
A_1 &= |\psi(\lambda_k)|^{-1} \left(
  \int_{-\pi}^{-\vartheta} + \int_\vartheta^\pi \right) \psi(\lambda)
  \expe^{\rmi j \lambda} D_{n,r}(\lambda_k - \lambda) d\lambda, \\ A_2 &=
  |\psi(\lambda_k)|^{-1} \psi^*(0) \int_{-\vartheta}^{\vartheta} (1 -
  \expe^{\rmi \lambda})^{-d} \expe^{\rmi j \lambda} D_{n,r}(\lambda_k - \lambda)
  d\lambda, \\ A_3 &= |\psi(\lambda_k)|^{-1} \int_{-\vartheta}^{\vartheta} (1 -
  \expe^{\rmi \lambda})^{-d} (\psi^*(\lambda) - \psi^*(0)) \expe^{\rmi j
  \lambda} D_{n,r}(\lambda_k - \lambda ) d\lambda.
\end{align*}
By Eq.~(\ref{eq:decaytaper}), if $|\lambda| \in [\vartheta, \pi]$,
$|D_{n,r}(\lambda_k-\lambda)| \leq C n^{-1/2-r}$. Note that $n^{-1} \lambda_k^{d} =
n^{-1} \lambda_k^{-1} \lambda_k^{d + 1} \leq 1/ (2 \pi k)$.  (\ref{eq:cond1bis})
implies that $ |A_1| \leq C n^{1/2-r} k^{-1} \leq C n^{-1/2}$.  Consider
$A_2$. Since $\int_{-\pi}^\pi D_{n,r}(\lambda) d\lambda = 0$,
\[
A_2 = \int_{-\vartheta}^\vartheta \Delta(\lambda,\lambda_k)
D_{n,r}(\lambda_k-\lambda) d\lambda, \quad \Delta(\lambda,\lambda_k) = \left( (1
- \expe^{\rmi \lambda})^{-d}  - (1 - \expe^{\rmi
\lambda_k})^{-d} \right)  \expe^{\rmi j \lambda}|\psi(\lambda_k)|^{-1}.
\]
Decompose this integral on the intervals $[-\vartheta,-\lambda_k/2]$,
$[-\lambda_k/2,\lambda_k/2]$, $[\lambda_k/2, 2 \lambda_k]$ and $[2 \lambda_k,
\vartheta]$. If $\lambda\in[-\lambda_k/2,\lambda_k/2]$, then
$|D_{n,r}(\lambda_k-\lambda)| \leq C \sqrt{n} k^{-r-1}$ and $ |
\Delta(\lambda,\lambda_k) | \leq C \left(|\lambda|^{-d} \lambda_k^{d} + 1
\right)$. Hence~:
\[
\left| \int_{-\lambda_k/2}^{\lambda_k/2} \Delta(\lambda,\lambda_k)
D_{n,r}(\lambda_k - \lambda) d\lambda \right| \leq C k^{-r} n^{-1/2}.
\]
If $\lambda\in[\lambda_k/2,2\lambda_k]$, then $|\Delta(\lambda,\lambda_k)| \leq
C \left( \lambda_k^{-1} |\lambda-\lambda_k| +1 \right)$.  Since
$\int_{-\lambda_k/2}^{\lambda_k} (1 + n |\lambda|)^{-r-1} d\lambda \leq C
n^{-1}$, we have
\[
\left| \int_{\lambda_k/2}^{2 \lambda_k} \Delta(\lambda,\lambda_k)\;
D_{n,r}(\lambda_k - \lambda) d\lambda \right | \leq C n^{-1/2}.
\]
If $\lambda \in [2\lambda_k, \vartheta]$ (and similarly on $[-\vartheta,
-\lambda_k/2]$), we use that $|\Delta(\lambda,\lambda_k)| \leq C (\lambda^{-d}
\lambda_k^{d}+1)$ and $|D_{n,r}(\lambda-\lambda_k)| \leq n^{-1/2-r}
|\lambda-\lambda_k|^{-1-r}$ . Hence,
\[
\left| \int_{2\lambda_k}^{\vartheta} \Delta(\lambda,\lambda_k) D_{n,r}(\lambda_k
- \lambda) d\lambda \right | \leq C n^{-1/2-r} \int_{\lambda_k}^\infty \left
(\lambda^{d} \lambda_k^{-d} + 1 \right) \lambda^{-1-r} d\lambda \leq C n^{-1/2}
k^{-r}.
\]
Consider $A_3$. Under~(\ref{eq:condloc}), we have
\[
| A_3 | \leq C \lambda_k^{d} \int_{-\vartheta}^\vartheta |\lambda|^{-d+\beta}
|D_{n,r}(\lambda-\lambda_k)| d\lambda.
\]
Decompose this integral as above. If $\lambda \in [-\lambda_k/2,\lambda_k/2]$,
proceeding as above:
\[
\lambda_k^{d} \int_{-\lambda_k/2}^{\lambda_k/2} |\lambda|^{-d+\beta}
|D_{n,r}(\lambda-\lambda_k)| d\lambda \leq C n^{-1/2} k^{-r} \lambda_k^\beta.
\]
If $\lambda_k \in [\lambda_k/2,2\lambda_k]$, $\lambda_k^{d}
|\lambda|^{-d+\beta} \leq C \lambda_k^\beta$, and
$\int_{\lambda_k/2}^{2\lambda_k} |D_{n,r}(\lambda-\lambda_k)| d\lambda \leq C
n^{-1/2}$.  Hence:
\[
\lambda_k^{d} \int_{\lambda_k/2}^{2\lambda_k} |\lambda|^{-d+ \beta}
|D_{n,r}(\lambda-\lambda_k)| d\lambda \leq C n^{-1/2} \lambda_k^\beta.
\]
Finally, if $\lambda \in [2\lambda_k,\vartheta]$ (and similarly, if $\lambda
\in [-\vartheta,-\lambda_k/2]$), we have as above:
\[
\lambda_k^{d} \int_{2\lambda_k}^\vartheta \lambda^{-d+\beta} \lambda^{-1-r}
d\lambda \leq C \lambda_k^{d} n^{-1/2-r} \int_{\lambda_k}^\vartheta
\lambda^{-d-1-r} d\lambda = C n^{-1/2} k^{-r}.
\]

\end{proof}
\begin{lem}
\label{lem:temporel}
There exists a constant $C$ (depending only on $\vartheta$, $\beta$, $\delta$,
$\Delta$, $\mu$,$r$) such that for all $k \in \{1,\dots,\tilde n\}$,
\begin{equation}
 \label{eq:temporel}
\frac1{\sqrt{nf(\lambda_k)}}\left| \sum_{t=1}^n h_{t,n}^r \psi_{t+j}
  \expe^{\rmi t \lambda_k} \right| \leq C n^{-1/2} \; \lambda_k^{d-1}
  (1+|j|)^{d-1} \leq C n^{-1/2} ((1+|j|)/n)^{d-1}.
\end{equation}
\end{lem}
\begin{proof}
  By applying the definition of the weights $h^r_{t,n}$ and summation by parts,
  we have:
\begin{align*}
  & \sum_{t=1}^n h_{t,n}^r \psi_{t+j} \expe^{\rmi t \lambda} = \sum_{p=0}^r
  (-1)^p \left(^r_p\right) \sum_{t=1}^n \psi_{t+j} \expe^{\rmi t
  (\lambda+\lambda_{p})},\\ & \sum_{t=1}^n \psi_{t+j} \expe^{\rmi t \lambda_k}
  = \sum_{t=1}^{n-1} \left \{ \left( \sum_{u=1}^t \expe^{\rmi u \lambda_k}
  \right) (\psi_{t+j} - \psi_{t+j+1}) + \left( \sum_{u=1}^n \expe^{\rmi u
  \lambda_k} \right) \; \psi_{n+j} \right \}.
\end{align*}
For all $y\in(0,\pi)$ and all $\ell\in\Nset^*$,
$ \left | \sum_{u=1}^\ell \expe^{ \rmi u y} \right| \leq 2/y
$. The proof follows from condition~(\ref{eq:cond2}).
\end{proof}
Proceed now with the proof of Theorem~\ref{theo:edgdftlong}.
If $|j|\geq n$, then $((1+|j|)/n)^{d-1}\leq 1$.  Hence by Lemma~\ref{lem:spectral}, for some constant $C$
which depends only on $\beta$, $\delta$, $\Delta$, $\vartheta$, $\mu$, $r$ and the
distribution of $Z_1$, $$\forall \ j,n,\bk, \ \   M_{n,j} \eqdef
Cn^{-1/2}\left( 1 \wedge ((1+|j|)/n)^{\delta-1}\right) \geq \|\bU_{n,j}(\bk)\|.$$ Note that
$$M_n \eqdef \sup_{j \in \Zset} M_{n,j} = C n^{-1/2}.
$$ Then \eqref{eq:edg:majorizingsequence} and \eqref{eq:edg:lindeberg} hold
uniformly in $\bk$.  By Lemma~\ref{lem:algebrique2}, Eq.~(\ref{eq:borneinfcov})
or (\ref{eq:borneinfsinglefreq}), we have
$$\sum_j \|\bU_{n,j}(\bk)\|^2 = \text{trace}[\bV_n(\bk)] \geq v_*>0.$$ Finally,
define for any $\gamma \geq 1$ the set $J_n =\{j \in \Zset, |j| \leq \gamma
n\}$. Then $\textrm{card}(J_n) \leq c_0 M_n^{-2}$ and
\[
\frac{\sum_{j \in \Zset \setminus J_n} \|\bU_{n,j}(\bk)\|^2}{\sum_{j \in \Zset}
  \|\bU_{n,j}(\bk)\|^2}
\leq
\frac{\sum_{j \in \Zset \setminus J_n} M_{n,j}^2}{\sum_{j \in \Zset}
  \|\bU_{n,j}(\bk)\|^2} \leq C (v_*)^{-1} n^{1 - 2\delta} \sum_{|j| \geq \gamma n}
  j^{2\delta-2} \leq C (v_*)^{-1}  \gamma^{2\delta-1}.
\]
Choosing $\gamma$ large enough yields~(\ref{eq:edg:smallness}) uniformly.

\section{Proofs of Corollaries~\ref{coro:boundmomshort1},
  \ref{coro:boundmomshort2}, \ref{coro:boundmomlong1} and
  ~\ref{coro:boundmomlong2}}
\label{sec:proofs-coroll}

\begin{proof}[{Proof of Corollary~\ref{coro:boundmomshort1}}]
By the triangle inequality, the LHS of inequality~(\ref{eq:boundmomshort1}) is bounded by
\begin{equation*}
  \left | \esp \left[ g(\bS_n(\bk)) \right] - \int_{\Rset^{2u}} g(\bx)
    \varphi_{\bV_n(\bk)}(\bx) \d \bx \right| +   \left | \int_{\Rset^{2u}} g(\bx)
    \left\{\varphi_{\bV_n(\bk)}(\bx) - \varphi_{\bV(\bk)}(\bx)\right\} \d \bx \right|.
\end{equation*}
By Corollary~\ref{coro:boundmom1} with $s=3$, the first term of the previous display is
bounded by $C n^{-1/2} N_3(g)$.   For $A$ a matrix, denote
$\rho(A)$ its spectral radius. Denote $\bI_{a}$ the $a$-dimensional identity
matrix.
To bound the second term, note that
$\rho(\bV_n(\bk) - \bV(\bk)) \leq C(\alpha,\beta)n^{-1}$
by~(\ref{eq:boundnormwnk}) and that $\tau(g,\bV(\bk))\geq 1$ by definition, then apply
the following lemma which is an easy adaptation of
\citet*[][Theorem~2.1]{soulier:2001}.
\end{proof}
\begin{lem}
\label{lem:soulier}
Let $\Gamma$ be a $u$-dimensional positive matrix.  There exists $\epsilon > 0$
and a constant $C$ such that, for all symmetric positive matrix $\Gamma'$ verifying
$\rho(\Gamma'^{-1}-\Gamma^{-1}) < \epsilon$, and for all measurable functions $g$ on
$\Rset^u$ satisfying $\| g \|^2_\Gamma < \infty$, we have
$$
\left| \int_{\Rset^u} g(\bx) \left\{ \varphi_{\Gamma'}( \bx) -
    \varphi_\Gamma(\bx) \right\} d \bx \right | \leq C \rho^{\tau(g,\Gamma)/2}
( \Gamma' - \Gamma) \; \| g \|_{\Gamma} \; .
$$
\end{lem}
\begin{proof}[{Proof of Corollary~\ref{coro:boundmomshort2}}]
  The LHS of (\ref{eq:boundmomshort2}) is bounded by $  A_1 + A_2 + A_3 + A_4$ with
\begin{align*}
  A_1 &= \left | \esp \left[ g(\bS_n(\bk)) \right] - \int_{\Rset^{2u}} g(\bx)
    \sum_{r=0}^{s-3} P_r(\bx,\bV_n(\bk),\{\chi_{n,\bnu}(\bk)\}) \}
    \d \bx \right| \; , \\
  A_2 &= \left | \int_{\Rset^{2u}} g(\bx) \left\{\varphi_{\bV_n(\bk)}(\bx) -
      \varphi_{\bI_{2u}/2}(\bx)\right\} \d \bx \right| \; ,  \\
  A_3 &= \left| \int_{\Rset^{2u}} g(\bx)
    P_1(\bx,\bV_n(\bk),\{\chi_{n,\bnu}(\bk)\}) \d \bx \right| \; ,  \\
  A_4 &= \left|\int_{\Rset^{2u}} g(\bx) \sum_{r=2}^{s-3}
    P_r(\bx,\bV_n(\bk),\{\chi_{n,\bnu}(\bk)\}) \} \d \bx \right| \; , 
\end{align*} $A_4 = 0$ if $s=4$.
Using~(\ref{eq:propG}), we get $\tau(g,\bI_{2u}/2) = 2$. It follows, as in the
proof of Corollary~\ref{coro:boundmomshort1} that $A_2$ is bounded by
$Cn^{-1}\|g\|_{\bV(\bk)}$, whereas $A_1$ is bounded by $Cn^{-(s-2)/2}N_s(g)$.
Write shortly  $P_r(\bx,\bV_n(\bk),\{\chi_{n,\bnu}(\bk)\}) =
R_r(\bx)\varphi_{\bV_n(\bk)}$,  where $R_r$ is a polynomial of order $r+2$ (the dependence w.r.t $\bV_n(\bk)$ and
 $\{\chi_{n,\bnu}(\bk)\}$ is ommited in this notation).
Note also that
\begin{multline}
  |\chi_{n,\bnu}(\bk)| \leq |\kappa_{|\bnu|}| \sum_{j \in \Zset}
  \|\bU_{n,j}(\bk)\|^{|\bnu|} \leq |\kappa_{|\bnu|}| M_n^{|\bnu|-2}(\sum_{j \in
    \Zset} \|\bU_{n,j}(\bk)\|^2) \\ \leq |\kappa_{|\bnu|}| M_n^{|\bnu|-2}
  \text{trace}[\bV_{n}(\bk)] \leq C|\kappa_{|\bnu|}| M_n^{|\bnu|-2}
  \label{eq:bornecum}
\end{multline}
where $M_n \leq C(\alpha,\beta) n^{-1/2}$.  Then, the coefficients of $R_r$ are
$O(n^{-(r/2)})$ uniformly in $\bk$ and $\psi$ since they involve
$\chi_{n,\bnu}(\bk)$'s with $|\bnu| = r$ and elements of $\bV_n^{-1}(\bk)$
\citep[for details, see][]{bhattacharya:rao:1976}.  Let $\bF_{n}(\bk) \eqdef
(\bV_{n}^{-1}(\bk) - \bV^{-1}(\bk))/2$ and write
 \begin{equation}
   \label{eq:expressiontermr}
   \int g(\bx) \varphi_{\bV_n(\bk)}(\bx) R_r(\bx) \d \bx = \left|\frac{\det
       \bV(\bk)}{\det \bV_{n}(\bk)}\right|^{1/2} \int g(\bx) R_r(\bx)
   \exp\{-\bx'\bF_n(\bk)\bx\}\varphi_{\bV(k)}(\bx) \d\bx
 \end{equation}
 By (\ref{eq:boundnormwnk}), $\|\bF_n(\bk)\|_1 \leq Cn^{-1}$ and $
 |\det(\bV_n(\bk))^{-1/2} - \det(\bI_{2u}/2)^{-1/2}| \leq C n^{-1}$ uniformly
 so that $A_4 \leq C n^{-1} \| (1+\|\bx\|^s)g(\bx)\|_{2\bI_{2u}/3}$.  We can
 derive this way that $A_3 \leq Cn^{-1/2} $ which is not enough. Improving
 this bound requires some care and uses the symmetries of $g$.  Actually, $R_1$
 is a sum of polynomials which are odd with respect to one or three components.
 Write
 \begin{align}
   \label{eq:taylor1}
   |\exp\{-\bx'\bF_n(\bk)\bx\} - 1 + \bx'\bF_n(\bk)\bx| \leq C n^{-2} \|\bx\|^4
   \exp \left\{ C n^{-1}\|\bx\|^2 \right\}
 \end{align}
 and notice that $ \{1 - \bx'\bF_n(\bk)\bx\} R_1(\bx)$ is a sum of polynomials
 of the form $\prod_i r_i(x_{2i-1},x_{2i})$, each of them being odd with
 respect to at least one variable. Consider a typical term odd with respect to
 $x_1$, say.   Using~(\ref{eq:propG})
\begin{multline*}
  \int_{\Rset^{2u}} g(\bx)\prod_i r_i(x_{2i-1},x_{2i})\varphi_{\bI_{2u}/2}(\bx)
  \d\bx = \int_{\Rset^2} g_1(x_1,x_2)r_1(x_1,x_2)\varphi_{\bI_2/2}(x_1,x_2)\d
  x_1 \d x_2 \\ \times \int_{\Rset^{2u-2}} \prod_{i>1}g_i(x_{2i-1},x_{2i})
  r_i(x_{2i-1},x_{2i}) \varphi_{\bI_{2u-2}/2}(\bx)\d x_3\cdots\d x_{2u} = 0,
\end{multline*}
since the first integral vanishes. Hence, $ \int_{\Rset^4} h(\bx) R_1(\bx)
(\bx'\bF_n(\bk)\bx) \varphi_{\bI_4}(\bx) \d\bx = 0.$ Gathering
(\ref{eq:bornecum}), (\ref{eq:expressiontermr}) and~(\ref{eq:taylor1}),  $A_3 \leq
Cn^{-2}$.
\end{proof}

\begin{proof}[Proofs of Corollaries~\ref{coro:boundmomlong1} and~\ref{coro:boundmomlong2}]
As those corollaries are the counterparts of
Corollaries~\ref{coro:boundmomshort1} and~\ref{coro:boundmomshort2} in a long
memory context,   we only gives the necessary adaptations from the preceding proofs.
  From Lemma~\ref{lem:covdft}, $\rho(\bV_n(k) - \bV(k)) \leq C p(k,k,n,\beta)$,
  $\|\bF_n(\bk)\|_1 \leq Cp(k,j,n,\beta)$ and
\begin{align*}
  |\det(\bV_n(\bk))^{-1/2} - \det(\bV(\bk))^{-1/2}| \leq C p(k,j,n,\beta).
\end{align*}
The LHS of (\ref{eq:taylor1}) is now bounded by
$p^2(k,j,n,\beta)  \|\bx\|^4 \exp \left\{ C p(k,j,n,\beta) \|\bx\|^2 \right\}$.
The term $A_3$ is then bounded by 
\begin{equation*}
  C n^{-1/2}p^2(k,j,n,\beta) \int_{\Rset^4} \|\bx\|^5
h(\bx)\exp\{-\|\bx\|^2(1+C p(k,j,n,\beta))\} \d\bx.
\end{equation*}
If $m=o(n)$ and $K_1 > 2C$, then for large enough $n$ and $K_1\leq k < j - r
\leq m - r$, the integral is uniformly bounded. Thus $ A_3 \leq C
n^{-1/2} p^2(k,j,n,\beta) $ whereas $A_1 \leq C  p^2(k,j,n,\beta) $.
\end{proof}

\section{Proof of Theorem~\ref{theo:rateoptimality}}
In the sequel, $C$ denotes a constant which depends only on $\beta$, $\delta$,
$\vartheta$, $\mu$ and the distribution of $Z_1$ and whose value may change
upon each appearance.  Note first that $ |\nu_k| = O(\log(k)), \ \ s_m^2/m \to
C >0$ (see for instance \citet*{robinson:1995a}, or
\citet*{hurvich:deo:brodsky:1998}).  Define $f^*(\lambda) = | 1 - \emi{\lambda}
|^{-2d}f(\lambda)$ and $L(\lambda) = \log(f^*(\lambda)/f^*(0))$.  Since
$\psi\in\mcf(\vartheta,\beta,\delta,\Delta,\mu)$, there exists a constant $C$ such
that
\begin{equation}
  \label{eq:lbeta}
\forall k \in \{1,\cdots,m\}, \ |L(\lambda_k)| \leq C |\lambda_k|^\beta.
\end{equation}
Let $\bar\eta$ denote $\esp(\log \|\bY\|^2)$ where $\bY$ is a centered Gaussian random
vector with covariance matrix $\bI_2/2$. Define $\eta_k =
\log(I_{k}/f(\lambda_{k})) - \bar\eta$, $1 \leq k \leq m$.  With these notations
and since $\sum_{k=1}^m \nu_k=0$, (\ref{eq:GPHexplicit}) yields
\begin{equation}
  \label{eq:decomphatd}
\hat d_m = d + s_m^{-2}\sum_{k=1}^m \nu_k \eta_k + s_m^{-2}\sum_{k=1}^m
\nu_k L(\lambda_{k}) =: d + W_m + b_m.
\end{equation}
The mean-square error of the GPH writes $\esp((\hat d_m - d)^2) = \esp W_m^2 +
2b_m\esp W_m + b_m^2$. Applying \eq{lbeta} and the Cauchy-Schwartz inequality,
\begin{equation}\label{eq:biais}
|b_m| \leq C s_m^{-2} \sum_{k=1}^m |\nu_k| \lambda_k^\beta \leq C (m/n)^\beta.
\end{equation}
Thus, to prove Theorem \ref{theo:rateoptimality}, we only need to show that
$\esp[W_m^2] \leq C m^{-1}$.  We now compute $\esp[W_m^2]$.  Let $\ell=\ell(m)$
be a non decreasing sequence of integers such that $1 \leq \ell \leq m$ and
define $W_{1,m} = s_m^{-2}\sum_{k=1}^\ell \nu_k \eta_k$ and
$W_{2,m}=W_m-W_{1,m}$.  We first give a bound for $\esp[W_{1,m}^2]$.  Note that
\begin{equation}
\label{eq:bornewm1}
\esp[W_{1,m}^2] \leq \ell s_m^{-4}\sum_{k=1}^\ell \nu_k^2 \esp[\eta_k^2].
\end{equation}
For $\bx\in\Rset^2$, define $g(\bx) = \log(\|\bx\|^2) -\bar\eta$. Then $\eta_k =
g(\bS_{n,k})$ and $N_3(g^2) < \infty$. For $(x_1,\dots,x_4) \in \Rset^2$,
define $h(x_1,\dots,x_4) = g(x_1,x_2)g(x_3,x_4)$. Then $\eta_k\eta_j =
h(\bS_{n,(k,j)})$, $h$ has property~(\ref{eq:propG}) and
\begin{equation*}
  N_{5}(h) = \int_{\Rset^4} \frac{g((x_1,x_2))g((x_3,x_4))}{1+\|\bx\|^5} \leq 4
  (N_{5/2}(g))^2 \d\bx < \infty
\end{equation*}
where we have used $4(1+(a^2+b^2)^{s/2}) \geq (1+|a|^{s/2})
(1+|b|^{s/2})$. Note that $N_4(h)=N_2(g) = +\infty$, which motivates the
expansion up to order $s=5$. Let $\sigma^2 \eqdef \var(\log
\|\bY\|^2)=\pi^2/6$. Applying Corollaries~\ref{coro:boundmomlong1}
and~\ref{coro:boundmomlong2} respectively to the functions $g,h$, we get for
some integer $l_0$ and any $k,j$ such that $l_0 \leq k < j \leq m$,
\begin{align}
  | \esp[\eta_k^2] - \sigma^2 | &\leq C(\beta,\delta,\vartheta,\mu)
  \left\{ k^{-1} + (k/n)^\beta + n^{-1/2} \right\}   \label{eq:varlogper}\\
    \label{eq:covlogper}
   \left| \esp[\eta_k \eta_j] \right| &\leq C(\beta,\delta,\vartheta,\mu)
   \left\{ k^{-2} + (j/n)^{2\beta} + n^{-1} \right\}.
\end{align}
(\ref{eq:bornewm1}) and~(\ref{eq:varlogper}) yield $\esp[W_{1,m}^2] \leq C
\ell^2 m^{-2}$.  We now bound $\esp[W_{2,m}^2]$:
\begin{align*}
  \esp[W_{2,m}^2] & = s_m^{-4}\sum_{k=\ell+1}^m \nu_k^2 \esp[\eta_k^2] + 2
  s_m^{-4} \sum_{\ell < k < j \leq m} \nu_k \nu_j \esp[\eta_k\eta_j].
\end{align*}
Using \eq{varlogper} and
\eq{covlogper}, we obtain
\begin{align}
  \left|\esp[W_{2,m}^2] - s_m^{-2}\sigma^2\right| \leq& C(\beta,\delta,\vartheta,\mu)
  s_m^{-4} \sum_{k=\ell+1}^m \nu_k^2 \left(k^{-1} + (k/n)^\beta + n^{-1/2}
  \right) \nonumber \\ & + C(\beta,\delta,\vartheta,\mu) s_m^{-4}\sum_{\ell < k
    < j \leq m} \nu_k \nu_j \left(k^{-2} + (j/n)^{2\beta} + n^{-1} \right)
  \nonumber \\ =& C(\beta,\delta,\vartheta,\mu) s_m^{-2} \left\{1 + O\left(
      \ell^{-1/2} + m^{1/2}l^{-3/2} + m^{2\beta+1}n^{-2\beta} + m/n
    \right)\right\}
  \label{eq:boundW2m}.
\end{align}
Choosing $\ell\leq m$ such that $\ell^2 = o(m)$ and $m = o(\ell^3)$ (for
instance $\ell = [m^\eta]$ with $1/3 < \eta < 1/2$) yields $\esp[W_m^2] =
O(m^{-1})$.  This bound and \eq{biais} conclude the proof of Theorem
\ref{theo:rateoptimality}.

\bibliographystyle{plainnat}
\bibliography{motherofallbibs,biblio}

\end{document}